\documentclass[12pt,reqno]{amsart}
\usepackage{xcolor,hyperref,tikz-cd,cancel}
\usepackage{setspace} % line distance
\usepackage{hyperref,soul}

\usepackage{caption}
\usepackage{xcolor}
\definecolor{marine}{RGB}{0,32,96}
\definecolor{bleudefrance}{rgb}{0.19, 0.55, 0.91}
\hypersetup{
	colorlinks,
	citecolor=bleudefrance,
	linkcolor=black,
	urlcolor=blue}

\usepackage[top=100pt,bottom=60pt,left=96pt,right=92pt]{geometry}
\usepackage{amsmath,enumerate,breqn}
 \usepackage{graphicx} 
   \usepackage{lscape}
\usepackage[framemethod=tikz]{mdframed}
\newcommand{\SU}{\operatorname{\mathsf{SU}}}  
\newcommand{\diam}{\operatorname{\mathsf{diam}}}  
\newcommand{\vol}{\operatorname{\mathsf{vol}}}  
\newcommand{\U}{\operatorname{\mathsf{U}}} 
 
\newcommand{\R}{\mathbb{R}}
   
\newcommand{\Diag}{\operatorname{\mathsf{Diag}}}  
\newcommand{\trace}{\operatorname{\mathsf{trace}}}
\newcommand{\Sp}{\operatorname{\mathsf{Sp}}}

\newcommand{\Spin}{\operatorname{\mathsf{Spin}}}

\newcommand{\C}{\mathbb{C}}
\newcommand{\Lie}{\mathsf{Lie}}   
\newcommand{\ad}{\operatorname{\mathsf{Ad}}}   
\numberwithin{equation}{subsection}
\usepackage{amssymb}  
\usepackage{tikz}
\usetikzlibrary{matrix,arrows.meta}
 
 \newtheoremstyle{mystyle}% name
 {}%Space above
 {}%Space below
 {\normalfont}%Body font
 {0pt}%Indent amount
 {\bfseries}% Theorem head font
 {.}%Punctuation after theorem head
 {4pt}%Space after theorem head 2
 {}%Theorem head spec (can be left empty, meaning ‘normal’)
 \theoremstyle{mystyle}
 \newtheorem{definition}{Definition}[section]
 
 \newtheorem{notation}[definition]{Notation}
 \newtheorem{remark}[definition]{Remark}

 \newtheorem{example}[definition]{Example}
 \newtheorem{examples}[definition]{Examples}
 
\newtheoremstyle{mystyle2}% name
{6pt}%Space above
{6pt}%Space below
{\itshape}%Body font
{0pt}%Indent amount
{\bfseries}% Theorem head font
{.}%Punctuation after theorem head
{4pt}%Space after theorem head 2
{}%Theorem head spec (can be left empty, meaning ‘normal’)
\theoremstyle{mystyle2}
\newtheorem{theorem}[definition]{Theorem}
\newtheorem{prop}[definition]{Proposition}

\newtheorem{lemma}[definition]{Lemma}
\newtheorem{corollary}[definition]{Corollary}

\newcommand{\purge}[1]{} % comments everything in {...} out, i.e.
\allowdisplaybreaks

\begin{document}

\title{Point vortex dynamics on K\"ahler twistor spaces}

%    Remove any unused author tags.

%    author one information
\author{Sonja Hohloch}
\address{\newline \textbf{Sonja Hohloch} \newline
    Department of Mathematics \newline
    University of Antwerp \newline
    Middelheimlaan 1 \newline 
    2020 Antwerp, Belgium}
\curraddr{}
\email{sonja.hohloch@uantwerpen.be}
\thanks{}

\author{Guner Muarem}
\address{\newline \textbf{Guner Muarem} \newline
    Department of Mathematics \newline
    University of Antwerp \newline
    Middelheimlaan 1 \newline 
    2020 Antwerp, Belgium}
\curraddr{}
\email{guner.muarem@uantwerpen.be}
\thanks{}

%\subjclass{test} displays old 1991 MSC classification


%\date{\today}
\dedicatory{}
\maketitle
\begin{abstract}
In this paper, we provide tools to study the dynamics of point vortex dynamics on $\mathbb{CP}^n$ and the flag manifold $\mathbb{F}_{1,2}(\C^3)$. These are the only K\"ahler twistor spaces arising from 4-manifolds. We give an explicit expression for Green's function on $\mathbb{CP}^n$ {which enables us to} determine the Hamiltonian $H$ and the equations of motions for the point vortex problem on $\mathbb{CP}^n$. Moreover, we determine the momentum map $\mu:\mathbb{F}_{1,2}(\C^3)\to \mathfrak{su}^*(3)$ on the flag manifold.
\end{abstract}

%%%%%%%%%%%%%%%%%%%%%%%%%%%%%%%%%%%%%%%
%%%%%%%%%% new section %%%%%%%%%%%%%%%
%%%%%%%%%%%%%%%%%%%%%%%%%%%%%%%%%%%%%
\section{Introduction}

\subsection{Point vortex dynamics}
The problem of dynamics of interacting point vortices goes back to the work of Helmoltz \cite{helmholtz1867lxiii} in the 19th century and can be formulated intuitively in its simplest form as follows.
Consider $N$ points $z_1, \dots, z_N$ (which we shall refer to as `vortices') in the plane $\C\simeq \R^2$ with coordinates $z_k = x_k + iy_k$. Let $\Gamma_1, \dots, \Gamma_N \in \R^{\neq 0}$ be real, non-zero numbers simulating the `vortex strength' of each point. The equations determining this dynamical system are given by the $N$ differential equations
 \[ \dot{\overline{z}}_j=\frac{1}{2\pi i}\sum_{k=1}^N \frac{\Gamma_k}{z_j-z_k} \qquad \mbox{for } 1 \leq j {  \neq k} \leq N. \]  
The signs of the vortex strengths determine the sense of rotation of each vertex.
 This system is in fact an example of a Hamiltonian system.
 Endow $\C^N$ with the symplectic form 
$\Omega:= \sum_{k=1}^N \Gamma_k \tau_k^* \omega_{st}$ where $\omega_{st}$ is the standard symplectic form on $\C$ and $\tau_k: \C^N \to \C$ the projection on the $k$th component.
Denote by $r (z_j , z_k ) := |z_j - z_k |$ the Euclidean distance between points $z_j, z_k \in \C \simeq \R^2$ and set
$$ 
\Diag_N(\C):= \{(z_1, \dots, z_N) \in \C^N \mid z_j = z_k  \mbox{ for some } 1 \leq j, k \leq N \mbox{ with } j \neq k\}.
$$
Abbreviate $z:=(z_1, \dots, z_N)$ and consider the Hamiltonian function given by 
 \begin{align*}
 %\label{planeHam}
 H: \C^N \setminus \Diag_N(\C) \to \R, \quad   H(z) : =-\frac{1}{4\pi}\sum_{k\neq j}\Gamma_j\Gamma_k\log(r(z_j,z_k)).
 \end{align*}
 Identifying $z_k = x_k + i y_k \simeq (x_k, y_k)$, the Hamiltonian equations are then given by
 \begin{align*}
 	\begin{cases}
 		\dot{x}_j=-\frac{1}{\Gamma_j}\partial_{y_j}H(x_1, \dots, x_N,y_1,\dots, y_N),
 	 \\	\dot{y}_j=\frac{1}{\Gamma_j}\partial_{x_j}H(x_1, \dots, x_N,y_1,\dots, y_N),
 	\end{cases}
 	\qquad \mbox{for } 1 \leq j \leq N.
 \end{align*}
Then one obtains the following three constants of motion (see for instance { Newton \cite{Newton2001}}) which reflect the invariance under translation and rotation of the system:
 \begin{align*}
 	p_x(z):=\sum_{j=1}^N\Gamma_j\Re(z_j), \qquad  p_y(z):=\sum_{j=1}^N\Gamma_j\Im(z_j) , \qquad m(z):=\frac{1}{2}\sum_{j=1}^N\Gamma_j|z_j|^2
 \end{align*}
where $\Re$ denotes the real part and $ \Im$ the imaginary part of a complex number. More generally, { Boatto \& Koiller \cite{Boatto2015}}
showed the following to be a fitting model for point vortex dynamics: Let $(M,\omega)$ be a symplectic manifold and consider the space $\mathcal{M}:= \Pi_{k=1}^N M \setminus \Diag_N(M)$ as phase space of $N$ moving vortices $z_1, \dots, z_N$ with vortex strengths $\Gamma_k\in\R^{\neq 0}$ for $1 \leq k \leq N$. 
Let $G$ be the fundamental solution (also called {\em Green's function}) of the Laplace-Beltrami operator and set 
\begin{align*}
R: M \to \R, \qquad R(z):=\lim_{s\to z}\left(G(s,z)-\frac{1}{2\pi}\log r(s,z) \right) 
\end{align*}
which is often referred to as {\em Robin function}.
The Hamiltonian of the system is then given by
\begin{align*}
	H:\mathcal{M} \to \R, \quad H(z_1, \dots, z_N ):=\sum_{1\leq j< k\leq N} \Gamma_i\Gamma_jG(z_j,z_k)+\sum_{k=1}^N\Gamma_k^2R_g(z_k).
\end{align*}
Green's function describes the interaction between pairs of distinct vortices and $R$ describes self-interactions of the vortices.  {Note that, on homogeneous manifolds, $R$ is  constant due to symmetry reasons and therefore often is neglected.}

%\par \textit{Quick remark: coadjoint orbits are K\"ahler manifolds; so in a sense I now restrict myself to the coadjoint orbits such that when one starts with the $4$-manifold $\mathbb{CP}^2$ from the classification (this is a degenerate coadjoint orbit) and constructs the twistor space $\mathcal{T}(\mathbb{CP}^2)$ one gets a K\"ahler twistor space. Moreover, this twistor space is exactly the flag manifold $\mathbb{F}_{1,2}(\mathbb{C}^3)$ being the generic coadjoint orbit. \\ \underline{Conclusion:} so in this case, twisting the degenerate $4$-dimensional coadjoint orbit of $\mathsf{SU(3)}$ gives the generic $6$-dimensional orbit. 

The equations of the point vortex problem are exactly the Euler equation arising in the discretization of fluid equations in mathematical modeling problems, see Aref \cite{aref2007point}, and Angrand \cite{angrand1985numerical} for the corresponding numerics.

During the past years, quite some work has been done on generalizing this approach to other symplectic manifolds, for example:
\begin{enumerate}
	\item  
	On the $2$-spheres $\mathbb{S}^2$ with $\mathsf{SO}(3)$-invariant symplectic form and Hamiltonian vector field, see Crowdy \cite{crowdy2006point}, Laurent-Polz $\&$ Montaldi $\&$ Roberts \cite{laurent2011point},  Lim $\&$ Montaldi $\&$ Roberts \cite{lim2001relative}.
	\item  
	There has also been some research done on the cylinder concerning periodic motion, see Montaldi $\&$ Souliere $\&$ Tokieda \cite{montaldi2003vortex}, Dritschel $\&$ Boatto \cite{dritschel2015motion}.
	\item 
	Point vortices on the cylinder, see Montaldi \& Souliere \& Tokieda \cite{montaldi2003vortex}.
	\item 
	Point vortices on the hyperbolic plane, see Montaldi $\&$ Nava-Gaxiola \cite{montaldi2014point}.
	\item 	
	Point vortices on $\mathbb{CP}^2$ with underlying symmetry group $\mathsf{SU}(3)$, see Montaldi $\&$ Shaddad \cite{montaldi2018generalized, Montaldi2019}. 
\end{enumerate}
A natural question is if the examples from above can be generalised to higher dimension. This naturally yields larger and more complicated symmetry groups. For example, in the case of the 2-sphere, the symmetry group is $G=\mathsf{SO}(3)$. 
But since the spheres $\mathbb{S}^{2n}$ for $n>1$ do not admit a symplectic structure, generalizing straightforward to higher dimensions with $\mathsf{SO}(m)$-symmetry does not necessarily make sense. 
Nevertheless, since $\mathbb{CP}^1\cong\mathbb{S}^2$ on can intuitively think of $\mathbb{CP}^n$ as the `best symplectic {analogue}’ of the $(n + 1)$-sphere, albeit with underlying higher dimensional symmetry group $\mathsf{SU}(n)$.

%%%%%%%%%%%%%%%%%%%%%%%%%%%%%%%%%%%%%%%%%%%%%%%
%%%%%%%%%%%%%%% new subsection %%%%%%%%%%%%%%%%

\subsection{Relation to twistor spaces}

%In the present paper, we consider a wider class of Hamiltonian systems with an action of the special unitary group $\mathsf{SU}(n)$. 
%Symplectic manifolds possessing $\mathsf{SU}(n)$-symmetry arise naturally when starting from $4$-manifolds in the following way. 
In the late seventies, 
Atiyah \cite{atiyah1978self} introduced the twistor theory for a 4-dimensional Riemannian manifold, relating it to 3-dimensional complex analysis. A few years later, in a paper by Hitchin \cite{hitchin1981kahlerian}, the question arose which complex manifold could be obtained by using Atiyah's twistor construction on compact $4$-manifolds. More precisely, the question was: which 4-manifolds have a twistor space which is K\"ahler? 

Surprisingly, there are not many, namely only the $4$-sphere $\mathbb{S}^4$ and the projective plane $\mathbb{CP}^2$ have K\"ahler twistor spaces. More specifically, the twistor space $\mathcal{T}(\mathbb{S}^4)$ is the complex projective space $\mathbb{CP}^3$ and $\mathcal{T}(\mathbb{CP}^2)$ is the 6-dimensional flag manifold (or Wallach space) $\mathbb{F}_{1,2}(\C^3) = \mathbb{W}^6=\mathsf{SU}(3)/\mathbb{T}^2$. 
\noindent{In this paper, we will be in particular interested in the spaces $\mathbb{CP}^n$ and the 6-dimensional flag manifold $\mathbb{F}_{1,2}(\C^3)$ in the context of point vortex dynamics.}  

%%%%%%%%%%%%%%%%%%%%%%%%%%%%%%%%%%%%%%%%%%%%%%%
%%%%%%%%% new subsection %%%%%%%%%%%%%%%%%%%%

\subsection{Main results}

One of the goals of this paper is to obtain an explicit expression for the Hamiltonian of the point vortex problem on certain coadjoint orbits in order to write down explicitly the equations of motion, look for conserved quantities, and analyse the underlying algebraic structure.

We are interested in symplectic manifolds with canonical $\mathsf{SU}(n)$-symmetry obtained by the coadjoint action of $\mathsf{SU(n)}$ on its dual Lie algebra $\mathfrak{su}(n)^*$. Specifically, we will focus on the case $n=3$. There exist exactly { three} coadjoint orbits, namely the six dimensional `generic' orbit
$$
\mathcal{O}^{\mathsf{SU}(3)}=\frac{\mathsf{SU}(3)}{\U(1)\times \U(1)}\cong \mathbb{F}_{1,2}(\C^3),
$$
the four dimension `degenerate' orbit (the meaning of this will become clear later on)
\[
\mathcal{O}_d^{\mathsf{SU}(3)}=\frac{\mathsf{SU}(3)}{\mathsf{SU}(2)\times \mathsf{SU}(1)}\cong \mathbb{CP}^2
\]
where the index $d$ refers to degenerate { and the trivial one.}

The first main question that we address in this paper arose from the following context: Montaldi $\&$ Shaddad \cite{montaldi2018generalized, Montaldi2019} studied the point vortex dynamics only on the degenerate orbit $\mathcal{O}_d^{\mathsf{SU}(3)}$ but not on the generic one. Thus a natural question is to investigate the point vortex problem on the generic orbit $\mathcal{O}^{\mathsf{SU}(3)}$.
Before we start with the associated momentum map we need some notation.
Let $B$ be the subgroup of upper triangular matrices of $\mathsf{SL}(3, \C)$. Then we may identify $\mathbb{F}_{1,2}(\C^3)\simeq \mathsf{SL}(3,\C)/ B$ (see Lemma \ref{lem:isoForFlag} for details) of which the elements are of the form
$$
\begin{pmatrix}
			1	&	0	&	0
		\\	z_1	&	1	&	0
		\\	z_2	&	z_3	&	1
		\end{pmatrix} 
		=:Z\in \mathsf{SL}(3,\C)/ B
$$
with $z_1, z_2, z_3 \in \C$.
Define the functions $K_1, K_2: \mathbb{F}_{1,2}(\C^3) \to \R$ by
$$
K_1 (Z): = 1 + |z_1|^2 + |z_2|^2  \quad \mbox{and} \quad K_2 (Z)  = 1 + |z_3|^2 + |z_1z_3-z_2|^2.
$$

\begin{theorem}
 \label{introTheoremA}
The momentum map of the left action of $\mathsf{SU}(3)$ on the generic coadjoint orbit $\mathcal{O}^{\mathsf{SU}(3)} = \mathbb{F}_{1,2}(\C^3)\simeq \mathsf{SL}(3,\C)/ B$ is explicitly given by
\begin{align*}
		\mu: \mathsf{SL}(3,\C)/ B \to \mathfrak{su}(3)^*,  \qquad
		\begin{pmatrix}
			1	&	0	&	0
		\\	z_1	&	1	&	0
		\\	z_2	&	z_3	&	1
		\end{pmatrix}  
		\mapsto (\mu_{ij})_{1 \leq i, j \leq 3}
	\end{align*}
	where $(\mu_{ij})_{1 \leq i, j \leq 3}$ is the traceless, anti-Hermitian matrix with entries
		\begin{align*}
	\mu_{11}&=\frac{1}{3} \left(\frac{x_3^2+y_3^2+2}{K_2}-\frac{x_2^2+y_2^2-1}{K_1}\right),
	\\
	\mu_{22}&=\frac{1}{3} \left(-\frac{2 x_2^2+2 y_2^2+1}{K_1}-\frac{x_3^2+y_3^2-1}{K_2}\right),
	\\ 
	\mu_{33}&=-(\mu_{11}+\mu_{22}),
	\\
	\mu_{12}&=\frac{\left(i y_1-x_1\right) \left(x_3-i y_3\right)-i y_2+x_2}{K_2}-\frac{x_1-i y_1}{K_1},
	\\
	\mu_{13}&=\frac{\left(i y_1-x_1\right) \left(x_3-i y_3\right)-i y_2+x_2}{K_2}-\frac{x_1-i y_1}{K_1},
	\\
	\mu_{23}&=\frac{i y_3+x_3}{K_2}-\frac{\left(x_1+i y_1\right) \left(x_2-i y_2\right)}{K_1}.
\end{align*}
The remaining entries are determined by the fact that the matrix is anti-Hermitian. 
\end{theorem}

Theorem \ref{introTheoremA} is restated as Theorem \ref{momentmap} and proven in Section \ref{sectionMomentumMap}.
We do not yet have an explicit formula for the Hamiltonian $H$ on $ \mathbb{F}_{1,2}(\C^3)$ since this requires an explicit expression of Green's function on $ \mathbb{F}_{1,2}(\C^3)$ which turned out to be more involved than expected and will be treated in a future work.

The second main question that we solved in this paper was motivated by the works of Montaldi $\&$ Shaddad \cite{montaldi2018generalized, Montaldi2019} about the dynamics of the generalised vortex problem on $\mathbb{CP}^2$ on which they worked without an explicit expression for the Hamiltonian. In fact, it is possible to compute the Hamiltonian of the point vortex problem for general $\mathbb{CP}^n$ explicitly. The answer involves first computing Green's function on $\mathbb{CP}^n$ (stated and proven as Theorem \ref{greenFunctionCPn} in Section \ref{sectionGreenAndHam}):

\begin{theorem}
 \label{introTheoremB}
	Consider $\mathbb{CP}^n$ with the Fubini-Study metric. Then Green's function is given by $G:\left(\mathbb{CP}^n\times \mathbb{CP}^n \right)\setminus \Diag_2( \mathbb{CP}^n) \to \mathbb{R}$ with
	\begin{align*}
	%G:\mathbb{CP}^n\times \mathbb{CP}^n\to \mathbb{R}
	%,\quad
	G(\xi, \eta)= 
	-\frac{1}{2n\cdot \vol(\mathbb{CP}^n)}\left(\log(\sin(r(\xi, \eta)))-\sum _{j=1}^{n-1} \frac{1}{2j \sin^{2j}(r(\xi, \eta))}\right)
	\end{align*}
	where 
	$r(\xi, \eta) = \arccos\sqrt{\frac{\langle \xi, \eta \rangle_H \langle \eta , \xi \rangle_H}{\langle \xi, \xi \rangle_H \langle \eta, \eta \rangle_H}}$ is the geodesic distance between two point in $\mathbb{CP}^n$ and $\langle \cdot,\cdot \rangle_H$ is the Hermitian inner product.
\end{theorem}

The following theorem is stated and proven as Theorem \ref{fullhamiltonian_cpn} in Section \ref{sectionGreenAndHam}.

\begin{theorem}
 \label{introTheoremC}
The Hamiltonian for the $N$ point vortex dynamics on the projective space $\mathbb{CP}^n$ is explicitly given by
\begin{align*}
& H: \ \bigl(\mathbb{CP}^n \bigr)^N \setminus \Diag_N(\mathbb{CP}^n) \ \to\ \R, \\
& H (\zeta) = -\frac{1}{2(n-1)!\pi^n}\sum_{\alpha < \beta}^N \Gamma_{\alpha}\Gamma_{\beta}\left(\log(\sin(r(\zeta_{\alpha},\zeta_{\beta})))-\sum _{j=1}^{n-1} \frac{1}{2j \sin^{2j}(r(\zeta_{\alpha},\zeta_{\beta}))}\right)
\end{align*}
where $\zeta = (\zeta_1, \dots, \zeta_N)$ and $r(\zeta_{\alpha},\zeta_{\beta})$ is the geodesic distance on $\mathbb{CP}^n$ between the two points given by 
\begin{align*}
r(\zeta_{\alpha},\zeta_{\beta}) = \arccos\sqrt{\frac{\langle \zeta_{\alpha},\zeta_{\beta}\rangle_H \langle \zeta_{\beta},\zeta_{\alpha}\rangle_H}{\langle \zeta_{\alpha},\zeta_{\alpha}\rangle_H \langle \zeta_{\beta},\zeta_{\beta}\rangle_H}},
\end{align*}
where $\langle \cdot,\cdot \rangle_H$ is the Hermitian inner product.
\end{theorem}

\noindent As already explained above, we do not yet have an explicit formula for the Hamiltonian $H$ on $ \mathbb{F}_{1,2}(\C^3)$ since we do not yet have an explicit expression of Green's function on $ \mathbb{F}_{1,2}(\C^3)$. The hope (see Section \ref{hamFlag}) is to make use of the fibration 
$$
\mathbb{S}^2\longrightarrow  \mathbb{F}_{1,2}(\C^3) \longrightarrow \C\mathbb{P}^2
$$
and to obtain Green's function on $\mathbb{F}_{1,2}(\C^3)$ from Green's functions on $\mathbb{S}^2$ and $\C\mathbb{P}^2$ and thus obtain the Hamiltonian on $\mathbb{F}_{1,2}(\C^3)$. This is planned to be carried out in a future work.

%%%%%%%%%%%%%%%%%%%%%%%%%%%%%%%%%%%%%%%%%
%%%%%%%%%% new subsection %%%%%%%%%%%%%%%

\subsection{Organisation of the paper} We give a quick overview of this paper:
\begin{itemize}
	\item {In Section \ref{sectionPreliminaries}, we recall necessary notions and results from Lie algebras, representation theory and differential geometry.}
	\item {In Section \ref{sectionCoadInfo}, we consider geometric and algebraic features and properties of the two coadjoint orbits $\mathcal{O}^{\SU(3)}_d \simeq \mathbb{CP}^2$ and $\mathcal{O}^{\SU(3)} \simeq \mathbb{F}_{1, 2}(\C^3)$.} 
	\item {In Section 4, we analyse the momentum map for various situations.}   
	\item {In Section \ref{sectionGreenAndHam}, we study Green's function in various settings and compute an explicit formula for it on $\mathbb{CP}^n$. After that, we compute the associated Hamiltonian function.}
\end{itemize}
%%%%%%%%%%%%%%%%%%%%%%%%%%%%%%%%%%%%%%%%%
%%%%%%%%%% new subsection %%%%%%%%%%%%%%%

%%%%%%%%%%%%%%%%%%%%%%%%%%%%%%%%%%%%%%%%%%%%%%%%%%%
%%%%%%%%%%%%%%% new section %%%%%%%%%%%%%%%%%%%
%%%%%%%%%%%%%%%%%%%%%%%%%%%%%%%%%%%%%%%%%%%%%%%

\section{Preliminaries}
\label{sectionPreliminaries}

%%%%%%%%%%%%%%%%%%%%%%%%%%%%%%%%%%%%%%%%
%%%%%%%%%%%% subsection %%%%%%%%%%%%%%
{ Throughout the whole paper, Lie groups are assumed to be compact unless stated otherwise.}
\subsection{Notions and conventions from group actions and Lie theory}
 Let $G$ be a compact
Lie group with Lie algebra $\Lie(G)=\mathfrak{g}$ and dual algebra $\mathfrak{g}^*$. 
The (left) \textit{action} of a Lie group $G$ on a manifold $M$ is denoted by the map $\Phi:G\times M\to M$ which satisfies for all $x\in M$
\begin{enumerate}
	\item $\Phi(e,x)=x$,
	\item $\Phi(g,\Phi(h,x))=\Phi(gh,x)$ for all $g,h\in G$.
\end{enumerate}
We usually write the action briefly as $g\cdot x$ or simply $gx$. 
Recall that Lie group actions are smooth.
Moreover, the \textit{isotropy subgroup} or \textit{stabilizer} of a point $m$ is given by the closed subgroup $G_m:=\{g\in G\mid gm=m\}$. 
The orbit under $G$ of a point $m \in M$ is given by
$$
\mathcal{O}_m:= \{ gm \in M \mid g \in G\}  \subseteq M.
$$ 
Lying in the same orbit gives rise to an equivalence relation on the manifold $M$ via $x\sim y\Leftrightarrow gx=y$ for $x,y\in M$ and $g\in G$. 
 The space consisting of all these equivalence classes is called the {\em orbit space} and denoted by $M/G$.
Now consider the action of a Lie group $G$ on itself by conjugation 
$$c_g:G\to G, \quad  h\mapsto ghg^{-1}.$$ 
Identifying the Lie algebra $\mathfrak{g}$ with the tangent space $T_eG$ at the neutral element $e \in G$ and differentiating $c_g$ in $e$, we obtain for all $g \in G$ 
\[ \ad_g:T_eG \simeq \mathfrak{g} \ \to \ T_eG \simeq \mathfrak{g}, \quad {\ad_gX=gXg^{-1}} \]
with adjoint representation
$$
\ad :G\times \mathfrak{g} \to \mathfrak{g} , \quad (g, X)\mapsto \ad_g(X).
$$
As dual notation, we have the \textit{coadjoint {representation}} 
$$\textsf{Ad}^*:G\times \mathfrak{g}^*\to \mathfrak{g}^*, \quad (g,\alpha)\mapsto \mathsf{Ad}^*_{g^{-1}}\alpha=g \alpha g^{-1}.$$
For every $\mu\in\mathfrak{g}^*$, the set 
\[ \mathcal{O}_{\mu} : =\{\ad_g^*\mu\mid \text{for all } g\in G\} \subseteq  \mathfrak{g}^* \]
is the \textit{coadjoint orbit} of $G$ through $\mu$. 

Note that, in general, the adjoint and coadjoint representations (and thus the resulting orbits) are \textit{not} isomorphic, see for instance counterexamples given by \textit{groups of Euclidean type} (cf.\ Arathoon $\&$ Montaldi \cite{arathoon2018hermitian}).

In this paper, we are working with Lie algebras that consist of matrices. Here the dual pairing between $\mathfrak{g}$ and $\mathfrak{g}^*$ is given by the so-called Killing form $\kappa ( \cdot, \cdot) $ which is a multiple of the trace of the product of the two matrices. For example, for $\mathfrak{su}(n)$, the Killing form is given by 
\begin{equation}
\label{ex:explicitTrace}
\kappa_{\mathfrak{su}(n)}(X,Y)=\trace(XY) \quad \mbox{for all } X,Y\in\mathfrak{su}(n).
\end{equation}
We denote the pairing between $\mathfrak{g}$ and $\mathfrak{g}^*$ by
$$
\langle\cdot,\cdot\rangle: \mathfrak{g}\to \mathfrak{g}^*, \quad X\mapsto \kappa_X \mbox{ where } \kappa_X(Y):=\kappa(X,Y) \mbox{ for all } X, Y \in \mathfrak{g}.
$$

%%%%%%%%%%%%%%%%%%%%%%%%%%%%%%%%%%%%%%%%%%%%%%%
%%%%%%%%%%% new subsection  %%%%%%%%%%%%%%%%%%%%

\subsection{Exponential of a matrix} 
In the context of Lie groups and Lie algebras, the exponential map is defined as 
\begin{align*}
\exp:\mathfrak{g}\to G,\quad X\mapsto \gamma(1)
\end{align*}
where $\gamma:\R\to G$ is the unique one-parameter subgroup of $G$ for which the tangent vector at the identity is $X$. In the case of a matrix Lie group, the exponential map is given by
\begin{align*}
	\exp(A) :=\sum_{k=0}^{\infty}\frac{1}{k!}A^k \quad \mbox {for all } (n\times n)\mbox{-matrices } A,
\end{align*}
briefly called the {\em exponential} of the matrix $A$.
If $A=\Diag(a_{11},\dots,a_{nn})$ is a diagonal matrix we obtain $A^k=\operatorname{Diag}(a^k_{11},\dots,a^k_{nn})$ and therefore
\begin{align*}
	\exp(A) & = \sum_{k=0}^{\infty}\frac{1}{k!}{\Diag}(a^k_{11},\dots,a^k_{nn})
	= \Diag\left( \sum_{k=0}^{\infty}\frac{1}{k!}a_{11}^k,\dots,  \sum_{k=0}^{\infty}\frac{1}{k!}a_{nn}^k\right) \\
	& =\Diag\left(e^{a_{11}},\dots,e^{a_{nn}}\right).
\end{align*}
If $A$ is diagonalisable with $A=PDP^{-1}$, where $P$ is the matrix of eigenvectors and $D$ is the diagonal matrix with the eigenvalues on the diagonal, then
\begin{align*}
	A^k= (PDP^{-1})^k=PDP^{-1} \cdots PDP^{-1}=PD^kP^{-1}
\end{align*}
and therefore
\begin{align}\label{exponential}
	\exp(A)=P\exp(D)P^{-1}.
\end{align}

%%%%%%%%%%%%%%%%%%%%%%%%%%%%%%%%%%%%%
%%%%%%%%%% subsection %%%%%%%%%%%%%%%

\subsection{Symplectic manifolds, Hamiltonian dynamics, and momentum maps}
A symplectic manifold $(M, \omega)$ is a smooth manifold $M$ equipped with a {\em symplectic form} $\omega$ which is a closed non-degenerate differential $2$-form, i.e.\ $d\omega=0$ and whenever $\omega_x(u,v)=0$ for all $u\in T_xM$ then $v=0$. This implies in particular that symplectic manifolds are always even dimensional.

Given a symplectic manifold $(M, \omega)$, the map 
$$
TM \to T^*M, \quad X \to \iota_X\omega \quad \mbox{where} \quad  \iota_X\omega(Y):= \omega(X, Y) \mbox{ for all } Y \in TM
$$
is an isomorphism, often referred to as {\em contraction} of $\omega$ by a vector field.
%Let $M$ be a manifold with vector field $X$. The interior product is defined as the operation which contracts a differential form $\omega$ via the vector field $X$. This contraction is denoted by $\iota_X\omega$. More formally, denote by $\Omega^p(M)$ the set of all $p$-forms on $M$, then the interior product is given by the following map:
%\begin{align*}
%	\iota_X:\Omega^p(M)\to \Omega^{p-1}(M), \quad \omega\mapsto \iota_X\omega,
%\end{align*}
%which satisfies the natural condition that \begin{align}
%	(\iota_X\omega)(X_1,\dots,X_{p-1}) = \omega(X_1,\dots,X_{p-1})
%\end{align}
%for all vector field $X_j$ with $j=1,\dots,p-1$ on $M$. Informally, said the interior product turns a $p$-form into a $(p-1)$-form by using a vector field. 
Symplectic manifolds are the natural geometric background for Hamiltonian dynamics:
Given a smooth function $H:M\to \R$, its {\em Hamiltonian vector field} (also called {\em symplectic gradient}) is defined via $\iota_{X_{H}}\omega=dH$. In this situation, $H$ is often referred to as {\em Hamiltonian function}.
%%%%%%%%%%%%%%
%\todo{\ssc{define almost complex structure, then compatible almost complex structure and then fix the rest of the text below}}
%
%In the case where the manifold $(M,\omega)$ has a compatible almost-complex structure $J$ there is the following formula relating the two gradients:
%
%\begin{align}
%\label{symplecticGradient}
%X^H :=\operatorname{grad}_s(H) = J\operatorname{grad}(H),
%\end{align}
%
%where $\operatorname{grad}(H)$ is the usual gradient and $\operatorname{grad}_s(H)$ the symplectic gradient. 
%%%%%%%%%%%%%
Let $G$ be a Lie group $G$ with Lie algebra $\Lie(G)=\mathfrak{g}$ and assume that the action $G\times M\to M$ is by symplectomorphisms, i.e, for all $g \in G$, the map $M \to M$, $x \mapsto g.x$ is a symplectomorphism. Denote by $\langle\cdot,\cdot\rangle$ the dual pairing $\langle\cdot,\cdot\rangle:\mathfrak{g}^*\times \mathfrak{g}\to \R$.  
Every $\xi\in\mathfrak{g}$ gives rise to a vector field $X_{\xi}$ via
\[ X_{\xi}(x)=\frac{d}{dt} {\bigg|_{t=0}} \exp(t \xi)\cdot x\] 
for all $x\in M$.
The {\em momentum map} for this $G$-action on $(M,\omega)$ is a map $\mu : M \to \mathfrak{g}^* $
such that 
\[ d(\langle \mu,\xi\rangle) = \iota_{X_\xi}  \omega \]
for all $\xi\in\mathfrak{g}$ where 
\[ \langle \mu,\xi\rangle : M\to \R , \quad x\mapsto \langle \mu(x),\xi\rangle. \]

%\todo{\ssc{I moved the example (= Example \ref{ex:exampleMomentumMap}) to where we later compute our momentum map; reason: this are no lecture notes, i.e., most readers should have some intuition about a momentum map and don't need an example here; but the example would be useful as `warm-up' later for the `real thing'...}}

%In order to compute this explicitly, we use the following trick. Consider a vector field $X$ on a manifold $M$ with local coordinates $(x_1,\dots,x_n)$ then the vector field $X$ can be written as a linear combination using the basis $(\partial_{x_1},\dots,\partial_{x_n})$. Dually, we have the basis $(dx_1,\dots,dx_n)$. In the case of a symplectic form, we can write down $\omega$ as a matrix. Then, $\iota_X\omega= X\omega X^*$ with respect to the basis above.

%\todo{\ssc{I moved the subsection on (the computation of) exact %1-forms to where we need it}}

%%%%%%%%%%%%%%%%%%%%%%%%%%%%%%%%%%
%%%%%%%%%%%%% subsection

\subsection{Weyl group and coadjoint orbits of $\mathsf{SU}(n)$} 
In this paper, we are in particular interested in the coadjoint orbits of the Lie group $\mathsf{SU}(3)$ since they are can be chosen to be the setting for vortex dynamics on the projective plane $\mathbb{CP}^2$ and the flag manifold $\mathbb{F}_{1,2}(\mathbb{C}^3)$. 

Let us start with fixing notation and recalling some properties of this Lie group and its Lie algebra. 
The general linear group is defined as \begin{align*}
\mathsf{GL}(n,\mathbb{C}) = \{A\in\operatorname{Mat}_n(\mathbb{C})\mid \det A\neq 0\}.
\end{align*}
The maximal compact simply connected Lie subgroup of $\mathsf{GL}(n,\mathbb{C})$ is given by \[ \mathsf{SU}(n) =\{U\in \operatorname{Mat}_n(\mathbb{C}) \mid U \overline{U}^T=1, \ \det U =1\} .\]
%where $U^\dagger:=\overline{U}^T$ is the conjugate transpose.
We have $\dim_\R \mathsf{SU}(n)= n^2-1$.
The Lie algebra $\mathfrak{su}(n)$ of $\mathsf{SU}(n)$ can be identified with 
$$
\{ 
U \in \operatorname{Mat}_n(\mathbb{C}) \mid \overline{U}^T= -U, \ \trace{U}=0
\} ,
$$
meaning all $(n\times n)$-matrices which are skew-Hermitian matrices with trace zero.  Our convention for the Lie bracket is $[A,B]=AB-BA$ for all $A, B \in \mathfrak{su}(n)$. We say that $A, B \in \mathfrak{su}(n)$ {\em commute} if $[A, B]=0$.
%Moreover, the complexification $\mathfrak{su}(n)^{\mathbb{C}}$ is isomorphic to $\mathfrak{sl}(n,\mathbb{C})$ and consists of all traceless $(n\times n)$-matrices.

%\todo{\ssc{do we need the content of the following orange text anywhere?}}
%{
%For the case $n=3$, the equations $U^{\dagger}U=1$ and $\det U =1$ are explicitly given by
%\begin{align*}
%	\begin{cases}
%		|u_{11}|^2 + |u_{21}|^2 + |u_{31}|^2 =1, 	\\
%		|u_{12}|^2 + |u_{22}|^2 + |u_{32}|^2=1, \\
%		|u_{13}|^2 + |u_{23}|^2 + |u_{33}|^2=1, 
%	\end{cases} \text{and}\quad
%	\begin{cases}
%		u_{11}u_{12}+ u_{21}u_{22} + u_{31}u_{32}= 0, \\
%		u_{11}u_{13}+ u_{21}u_{23} + u_{31}u_{33}=0, \\
%		u_{12}u_{13}+ u_{22}u_{33} + u_{32}u_{33}=0.
%	\end{cases}
%\end{align*}
%}
In what follows, we will often work with the following basis of $\mathfrak{su}(3)$:

\begin{notation}
\label{rescaledGel} % was vroeger een lemma
  	The following eight traceless traceless $(3\times 3)$-matrices are known as the Gell-Mann matrices. 
\begin{align*}
\tilde{\lambda}_1 &= {\begin{pmatrix}0&1&0\\1&0&0\\0&0&0\end{pmatrix}}          &  \tilde{\lambda}_2 &={\begin{pmatrix}0&-i&0\\i&0&0\\0&0&0\end{pmatrix}}              &  \tilde{\lambda}_3&={\begin{pmatrix}1&0&0\\0&-1&0\\0&0&0\end{pmatrix}}		&	\tilde{\lambda}_4&={\begin{pmatrix}0&0&1\\0&0&0\\1&0&0\end{pmatrix}}\\
\tilde{\lambda}_5&={\begin{pmatrix}0&0&-i\\0&0&0\\i&0&0\end{pmatrix}}         &  \tilde{\lambda}_6&={\begin{pmatrix}0&0&0\\0&0&1\\0&1&0\end{pmatrix}}   &  \tilde{\lambda}_7&={\begin{pmatrix}0&0&0\\0&0&-i\\0&i&0\end{pmatrix}}		&	\tilde{\lambda}_8&={\frac{1}{\sqrt{3}}}{\begin{pmatrix}1&0&0\\0&1&0\\0&0&-2\end{pmatrix}}.
\end{align*}
We have $[\tilde{\lambda}_3,\tilde{\lambda}_8]=0$ and no other of these matrices commute with both $\tilde{\lambda}_3$ and $\tilde{\lambda}_8$.
The set 
$$\left\{ \left. \lambda_k:= \frac{i}{2}\tilde{\lambda}_k \  \right| \ k=1,\dots,8\right\}$$ 
forms a (rescaled) basis for the Lie algebra $\mathfrak{su}(3)$, often called {\em Gell-Mann basis}.
\end{notation}

\begin{remark}
\label{rem:casimir}
{The Lie algebra $\mathfrak{su}(3)$ has rank two. We denote the corresponding Casimirs living in the universal enveloping algebra by $C_1, C_2 \in \mathcal{U}(\mathfrak{su}(3))$, i.e.\ we have
	$[C_1,\lambda_k]=[C_2,\lambda_k]=0$ for all $k=1,\dots,8$. They are explicitly given by $C_1 = \sum_{k=1}^8 \lambda_k \lambda_k $ and $C_2 = 8 \sum_{j, k, \ell=1}^8 d_{j k \ell} \lambda_j \lambda_k \lambda_\ell$ where $d_{j k \ell} $ are the so-called structure constants of $\mathfrak{su}(3)$. }
\end{remark}

We will now recall the so-called {\em Weyl group}.  
Let $E$ be a finite dimensional vector space over $\mathbb{R}$ and let $\langle \cdot,\cdot\rangle_E$ be an inner product on $E$.
A {\em roots system} $\Phi\subset E$ is a finite set of non-zero vectors, called \textit{roots} such that:
\begin{enumerate}
	\item 
	The set of roots $\Phi$ spans the space $E$.
	\item 
	If $\alpha\in\Phi$ and $c\in\R$, then $c\alpha\in\Phi$ if and only if $c=\pm 1$.
	\item  $\sigma_{\alpha}(\beta):=\beta-2\frac{\langle \alpha,\beta\rangle_E}{\langle \alpha,\alpha \rangle_E}\alpha\in\Phi$ for all $\alpha, \beta \in \Phi$, 	
	i.e. $\Phi$ is invariant under $\sigma_\alpha$ for all $\alpha \in \Phi$ which is the reflection about the hyperplane orthogonal to $\alpha$. 
	\item
	$( \alpha,\beta) :=2\frac{\langle \alpha,\beta\rangle_E}{\langle \alpha,\alpha\rangle_E} \in \mathbb{Z}$ for all $\alpha, \beta \in \Phi$, i.e.  the projection of $\beta$ onto the line through $\alpha$ is an integer or half-integer multiple of $\alpha$.
\end{enumerate}
A subset $\Phi^+ \subset \Phi$ is called a {\em positive root system} if
\begin{enumerate}
	\item 
	for all $\alpha \in \Phi$, either $\alpha \in \Phi$ or $- \alpha \in \Phi$,
	\item
	for all $\alpha, \beta \in \Phi^+$, we have $\alpha + \beta \in \Phi^+$.
\end{enumerate}
Denote by 
$\mathsf{O}(E):=\{A\in \mathsf{GL}(E)\mid \langle Av,Aw\rangle_E =\langle v,w\rangle_E, \text{ for all } v,w\in E\}$ 
the orthogonal group consisting of all elements in $E$ preserving the inner product. The (finite) subgroup $\mathcal{W} \leq \mathsf{O}(E)$ generated by all reflections $\sigma_{\alpha}$ with $\alpha \in \Phi$ is called the {\em Weyl group} associated to $\Phi$.
Denote the hyperplane perpendicular to $\alpha\in\Phi$ by $\Pi_\alpha$. 
The closure of a  connected component of $E \setminus \{ \Pi_\alpha \mid \alpha  \in \Phi\}$ is called a {\em Weyl chamber}. 
%A Weyl chamber $W$ is called {\em positive} if, for all $w \in W$, we have $\langle w, \alpha \rangle \geq 0$ for all $\alpha \in \Phi^+$.
We define the {\em positive Weyl chamber} (with respect to a fixed choice of $\Phi^+$) as the closed set \[ \mathcal{C} = \{x\in E\mid \langle x,\alpha\rangle_E\geq 0 \text{ for all $\alpha\in\Phi^+$}\}. \]
%In order to speak of symmetries of a Hamiltonian system, we need the following notion of an action which in essence describes the compatibility between the Lie structure (encoding the symmetries) and the dynamics. 
%We say that the action $\Phi:G\times M\to M$ is \textit{canonical} if \[ \Phi_g^*\{f,g\}=\{\Phi_g^*f,\Phi_g^*g\}\qquad\Phi_g^*\omega=\omega. \]
%One can now speak of a $G$-symmetric Hamiltonian system if $G$ acts canonically on $M$ and if the Hamiltonian function $H$ is $G$-invariant, in a sense that $\Phi_g^*H=H$.
%\par On the infinitesimal level, the Hamiltonian system is $\mathfrak{g}$-symmetric if the Lie algebra $\mathfrak{g}$ acts canonical on $M$, meaning that the Lie derivative of the symplectic form vanishes, i.e. $\mathcal{L}_{\xi_M}\omega=0$ and if $H$ is $\mathfrak{g}$-invariant.
%
Given a positive root system, there is only one positive Weyl chamber.
Let $G$ be a {compact} Lie group and consider a maximal torus $H\subset G$ (i.e.\ a compact, connected, abelian Lie subgroup of $G$ which is maximal with respect to these properties) and the Cartan algebra $Lie(H):=\mathfrak{h}\subset \mathfrak{g}$ with dual $\mathfrak{h}^*$. Note that a maximal torus is unique up to conjugation. In the case of $\mathsf{SU}(n)$ the situation is as follows:

\begin{example}
The maximal torus $T$ 
of $\SU(n)$ is given by the diagonal matrices $\Diag(e^{i\theta_1},\dots,e^{i\theta_n})$ such that $\prod_{j=1}^n e^{i\theta_j}=1$. The Lie algebra $\Lie(T)=:\mathfrak{t}$ (i.e.\ the Cartan algebra)
is then given by the space of traceless diagonal matrices 
$\mathfrak{t}=\left\{\Diag(\theta_1,\dots,\theta_n) \left|\ \sum_{j=1}^n \theta_j=0\right.\right\}$.
Thus, the {interior $\mathfrak{t}_+^0$ of the} positive Weyl chamber $\mathfrak{t}_+$ is given by \[ \mathfrak{t}_+=\left\{(x_1,\dots,x_n)\in \R^n\left |\  x_1> x_2> \cdots > x_n\text{ and } \sum x_i=0\right.\right\} \]
and the {closure $\overline{\mathfrak{t}}_+ = \mathfrak{t}_+$} is given by replacing $>$ by $\geq$ in the above set. { In order to make a visualisation, we will now focus on the case $n=3$. Then} the Weyl group is the symmetric group $\mathsf{Sym}(3)$ generated by the positive roots $\alpha$, $\beta$, $\alpha + \beta$ sketched in Figure \ref{fig:root}. It permutes in fact all roots.
\end{example}

\begin{figure}[h!]
	\centering
	\includegraphics[width=0.45\linewidth]{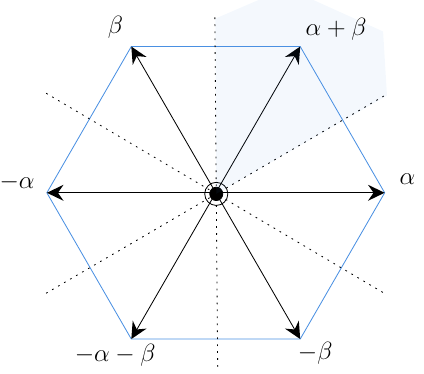}
	\caption{Root diagram of $\mathfrak{su}(3)$ where the positive roots are given by $\alpha = (1,0)$, $\beta=\left(	-\frac{1}{2}, \frac{\sqrt{3}}{2}	\right)$ and $\alpha+\beta=\left(	\frac{1}{2}, \frac{\sqrt{3}}{2}	\right)$. The simple roots are given by $\alpha$ and $\beta$. The blue shaded area is the positive Weyl chamber.} 
	\label{fig:root}
\end{figure}

The orbits of the Weyl group are described as follows (see e.g.\ Kirillov \cite{Kirillov2004} for a proof).

\begin{theorem}
Let $G$ be a Lie group with Lie algebra $\mathfrak{g}$ and Cartan subalgebra $\mathfrak{h}$ and dual $\mathfrak{h}^*$. Let $\mathcal{O}_{\mu}$ be a coadjoint orbit of $G$. Then $\mathcal{O}_{\mu}\cap \mathfrak{h}^*$ is an orbit of the Weyl group.
\end{theorem}
Here we see $\mathfrak{h}^*$ a subspace of $\mathfrak{g}^*$. For details, see Guillemin \& Sternberg \cite{MR0770935}.
This implies that each coadjoint orbit $\mathcal{O}$ of $G$ is uniquely defined by a starting point $\mu_0\in\mathfrak{h}^*$ in the (closed) positive Weyl chamber (for more details, see for example Bernatska $\&$ Holod et al.\ \cite{bernatska2008geometry}).

\begin{definition}
Denote the interior of the positive Weyl chamber $W$ by $W^{\circ}$ and its boundary by $\partial W$. The coadjoint orbit of a point $\mu_0\in W^{\circ}$ is said to be a {\em generic} orbit. The coadjoint orbit of a point $\mu_0\in\partial W$ is called a {\em degenerate} orbit.
\end{definition}

\begin{example}
\label{ex:coadjointOrbits}
The group $\SU(3)$ has exactly two coadjoint orbits, a generic one of dimension six and a degenerate one of dimension four. 
The generic orbit is denoted by $\mathcal{O}^{\mathsf{SU}(3)}$ and can be identified with
$$
\mathcal{O}^{\mathsf{SU}(3)}=\frac{\mathsf{SU}(3)}{\U(1)\times \U(1)}.
$$
The degenerate orbit is denoted by $\mathcal{O}^{\mathsf{SU}(3)}_d$ and can be identified with
$$
\mathcal{O}_d^{\mathsf{SU}(3)}=\frac{\mathsf{SU}(3)}{\mathsf{SU}(2)\times \mathsf{SU}(1)}\cong \mathbb{CP}^2.
$$
\end{example}

%%%%%%%%%%%%%%%%%%%%%%%%%%%%%%%%%%%%%%%
%%%%%%%%%% new section %%%%%%%%%%%%%%%
%%%%%%%%%%%%%%%%%%%%%%%%%%%%%%%%%%%%%%%

\section{Geometric structures of coadjoint orbits of $\mathsf{SU}(3)$}
\label{sectionCoadInfo}
%\ggc{In order to introduce the phase space and related dynamics with underlying Lie group $\SU(3)$ it is important to grasp back at the theory of Casimirs and their interpretation as constants of motion in the dynamical context. In this section we will mostly follow [Perelmov etc.] for the physical background. We start by considering classical observables $a_{ij}$ (they thus discribe a classical Hamiltonian system) for $i,j=1,2,3$ submitted to the relations
%	\begin{align}\label{relations}
%	a_{ij}=\overline{a_{ji}}
%	\\ a_{11}+a_{22}+a_{33}=0
%	\end{align}
%	Furthermore, we assume that the Poisson bracket give rise to a copy of $\SU(3)$:
%	\[ \{a_{kl},a_{mn}\} = i(\delta_{ml}a_{kn}-\delta_{kn}a_{ml}). \] The Hamiltonian dynamics of the system is now governed by the equation \[ \dot{s}_{ij} = \{H,a_{ij}\},\] where $H$ is the Hamiltonian function.  The relations (\ref{relations}) suggest that the observables $a_{ij}$ can be thought of as the entries of $3\times 3$ traceless Hermitian matrix $S$. The following two Casimir functions are constants of motion of the Hamiltonian $H(a_{ij})$
%	\begin{align*}
%	C_1 & = \trace S^2 = \sum_{k,l=1}^3s_{kl}s_{lk}
%	\\ C_2 & = \trace S^3 = \sum_{k,l,m=1}^3 s_{kl}s_{lm}s_{mk}.
%	\end{align*}}

In this section, we characterize coadjoint orbits of $\mathsf{SU}(3)$ by their algebraic and geometric properties and, eventually, describe their K\"ahler structure.

%%%%%%%%%%%%%%%%%%%%%%%%%%%%%%%%%%%
%%%%%%%%% new subsection %%%%%%%%%%%%%%

\subsection{Coadjoint orbits characterized by eigenvalues}
Let $(X,\{\cdot,\cdot \})$ be a Poisson manifold. A maximal connected submanifold $Y\subset X$ for which the Poisson structure descends to a symplectic structure is called a \textit{symplectic leaf}. Moreover, the Poisson manifold is foliated by its symplectic leaves.
Let $\mathfrak{g}$ be a Lie algebra with dual $\mathfrak{g}^*$. Then there is a canonical Poisson structure on $\mathfrak{g}^*$ called the Lie-Poisson structure.  In this case, the symplectic leaves are the coadjoint orbits.
A smooth function $C:\mathfrak{g}\to\R$ is called a {\em Casimir function} if $C$ is constant on each coadjoint orbit, or equivalently, if $C$ is invariant under the coadjoint action of $G$ on $\mathfrak{g}^*$. { We further recall that the \textit{joint} level sets of Casimir functions $C:\mathfrak{g}^*\to \mathbb{R}$ are symplectic manifolds (as they are coadjoint orbits).}

In the case of $\mathfrak{su}(n)$, the Casimir functions are given by \[ C_j:\mathfrak{su}^*(n)\to \R, \quad A\mapsto \trace(A^j) \] for $j=1,\dots,n-1$ where $A^j:= A \circ \dots \circ A$ is the $j$-fold composition.
For $n=3$, we may therefore study the dynamics on the intersection of the Casimir level sets $C_1=c_1$ and $C_2=c_2$ for some constants $c_1$ and $c_2$.
The set $\{C_1=c_1\} \cap \{C_2 = c_2\}$ can be identified with the space
\begin{equation*}
%\label{eq:spaceM}
\mathcal{M} := \{	A\in \operatorname{Mat}_3(\C) \mid A= \overline{A}^T, \ \trace(A) = 0, \ \trace(A^2) = c_1	, \ \trace(A^3)=c_2 	\}
\end{equation*}
which can be seen as phase space of a $\SU(3)$-invariant dynamical system.
In the generic case, we have $\dim(\mathcal{M})=6$ and, in the degenerate case, $\dim(\mathcal{M})=4$ which corresponds to the dimensions of two coadjoint orbits of $\mathsf{SU}(3)$, see Example \ref{ex:coadjointOrbits}. 

\begin{lemma}
The phase space $\mathcal{M}$
%defined in \eqref{eq:spaceM} 
is isomorphic as vector space to the space
$$
\{A   \in \operatorname{Mat}_3(\C) \mid A=\overline{A}^T,  A \text{ has eigenvalues } \lambda_1 \geq \lambda_2 \geq \lambda_3\text{ with } \lambda_1+\lambda_2+\lambda_3=0	\}.
$$
		%\begin{enumerate}[\normalfont(i)]
			%\item $	\{A\in \mathcal{H}_3 \mid \trace(A) = 0, \trace(A^2) = c_1	, \trace(A^3)=c_2\}$,
			%\item $\{	A\in \mathcal{H}_3\mid \text{eigenvalues } \lambda_1,\lambda_2,\lambda_3\text{ such that } \lambda_1+\lambda_2+\lambda_3=0	\}$.
		%\end{enumerate}
	\end{lemma}
	\begin{proof}
		{Recall that the characteristic polynomial of a $3\times 3$-matrix $A$ is given by
		\begin{align*}
		\frac{1}{6}\left(\trace^3(A)+2\trace(A^3)-3\trace(A)\trace(A^2)\right)\\-\frac{1}{2}\left(\trace^2(A)-\trace(A^2)\right)\lambda+\trace(A)\lambda^2-\lambda^3.
		\end{align*}}
Thus the characteristic polynomial $\chi_A$ of a traceless Hermitian matrix $A$ is of the form:
	\begin{align*}
		\chi_A=\det(A-\lambda I_3) 	&=-\lambda^3-\frac{1}{2}\trace(A^2)\lambda+\frac{1}{3}\trace(A^3)
		\\&=-\lambda^3+(\lambda_1+\lambda_2+\lambda_3)\lambda^2-(\lambda_1\lambda_2+\lambda_1\lambda_3+\lambda_2\lambda_3)\lambda+\lambda_1\lambda_2\lambda_3
		\\&=(\lambda_1-\lambda)(\lambda_2-\lambda)(\lambda_3-\lambda)
	\end{align*}
		
		{ Now note that $\trace(A) = 0, \trace(A^2) = c_1$ and $\trace(A^3)=c_2 $ from the original definition of $\mathcal{M}$ corresponds to 
		$A$ having (ordered) eigenvalues $\lambda_1\geq \lambda_2\geq \lambda_3$ with $\lambda_1+\lambda_2+\lambda_3=0$.} 
    \end{proof}
Let $\Lambda:=\Diag(\lambda_1,\lambda_2,\lambda_3)$ with $(\lambda_1,\lambda_2,\lambda_3)\in\R^3$ and set $\mathcal{O}_{\Lambda}=\{A\Lambda A^{-1}\mid A\in \SU(3)\}$.
The spectral theorem for Hermitian matrices states that the eigenvalues of a Hermitian matrix are real and that the eigenvectors corresponding to these eigenvalues are orthogonal. 
This allows to deduce the following bijective correspondence:
\begin{align*}
\begin{Bmatrix} \text{coadjoint orbits}\\\text{of } \mathsf{SU}(3) \end{Bmatrix}
&	\longleftrightarrow 
\begin{Bmatrix} \mathcal{O}_{\Lambda}\text{ with } \lambda_1+\lambda_2+\lambda_3=0 \\ \text{and } \lambda_1\geq \lambda_2\geq\lambda_3\end{Bmatrix}
\end{align*}
%\begin{example}
%	Consider the matrix
%	\[ \left(
%	\begin{array}{ccc}
%	2 & 2 i & -i \\
%	-2 i & 4 & -1 \\
%	i & -1 & -6 \\
%	\end{array}
%	\right) \]
%	it's characteristic polynomial is given by \[ \chi(\lambda) = -\lambda^3 + 34 \lambda - 26 \]
%	and $\lambda_1+\lambda_2+\lambda_3=0$.
%\end{example}
This leads to three types orbits (of which one is trivial):
\begin{enumerate}[(i)]
	\item {\em All three eigenvalues are distinct.} Then the stabilizer is given by $\Diag(\alpha,\beta,\overline{\alpha\beta})$ with $\alpha, \beta \in \C$ and the coadjoint orbit is
	\[ \mathcal{O}^{\SU(3)}=\frac{\SU(3)}{\U(1)\times\U(1)} =\frac{\U(3)}{\U(1)\times \U(1)\times \U(1)}. \] 
	We will see in Section \ref{sec:flagManifolds} that this orbit can be identified with a six dimensional flag manifold.
	
	\item {\em Two eigenvalues are equal.} The stabilizer is given by the block diagonal matrix $\Diag(A,\overline{\det A})$ where $A\in\U(2)$. In this case the coadjoint orbit can be identified with
	\[ \mathcal{O}^{\SU(3)}_d =\SU(3)/\U(2)\cong \mathbb{CP}^2. \]
	\item {\em All eigenvalues are equal}.  In this case, the stabilizer is $\SU(3)$ so that the orbit is trivial (since $\lambda_1+\lambda_2+\lambda_3=3\lambda=0$ implies $\lambda=0$). 
\end{enumerate}

\begin{remark}
	More generally, setting $\Lambda:=\Diag(\lambda_1,\dots,\lambda_n)$ with $(\lambda_1,\dots,\lambda_n) \in \R^n$ and $\lambda_1\geq\cdots\geq \lambda_n$ and $\sum_{i=1}^n \lambda_i=0$, the coadjoint orbits of $\SU(n)$ are of the form $\mathcal{O}_{\Lambda}=\{A\Lambda A^{-1}\mid A\in \SU(n)\}$.
\end{remark}

%\begin{definition} 
%	A \textit{fiber bundle} is a structure  $(E,B,\tau ,F)$, where $B$, and 
%	$F$ are topological spaces and  ${\displaystyle \tau :E\rightarrow B}$ is a continuous surjection satisfying a local triviality condition outlined below. The space $B$
%	is called the base space of the bundle,  $E$ the total space, and 
%	$F$ the fiber. The map $\tau$ is called the projection map (or bundle projection). We shall assume in what follows that the base space 
%	$B$ is connected. We say that $E$ is a fibre bundle over the base space $B$.
%	\par For every $x\in B$ there is a neighbourhood $U\in\mathcal{V}(x)\subset B$ such that the following diagram commutes:
%\end{definition}

%%%%%%%%%%%%%%%%%%%%%%%%%%%%%%%%%%%%%%%%%
%%%%%%%%% subsection %%%%%%%%%%%%%%%%%%%%

\subsection{Coadjoint orbits seen as flag manifolds}

\label{sec:flagManifolds}

The four dimensional degenerate orbit space has a nice geometrical interpretation as $\mathbb{CP}^2$. We will now see that there is also a nice geometric characterisation of the generic orbit as a so-called flag manifold (of which the degenerate orbit $\mathbb{CP}^2$ is a special case).

\begin{definition}
Consider $\C^n$ and let $r \in \{1, \dots, n\}$.
A {\em flag} $f_{{k_1},\dots,{k_r};n}$ in $\C^n$ is a nested sequence of vector subspaces $V_{k_1}\subsetneq \dots \subsetneq V_{k_r}$ in $\C^n$ such that $\dim_{\C}V_{k_j}=k_j$ for all $1 \leq j \leq r$. The space of all such flags is denoted by $\mathbb{F}_{{k_1},\dots,{k_r}}(\C^n)$.
\end{definition}

\begin{remark}
$\mathbb{F}_{{k_1},\dots,{k_r}}(\C^n)$ is a compact, complex and smooth manifold and is usually referred to as the \textit{flag manifold}. 
Note that all flag manifolds are in fact generalisations of projective spaces. The flag manifold $\mathbb{F}_{1}(\C^n)$ is precisely $\C\mathbb{P}^{n-1}$.  
Moreover, the flag manifold $\mathbb{F}_{k}(\C^n)$ is the space of $k$-dimensional vector subspaces of $\C^n$, i.e., the Grassmannian.
\end{remark}

\noindent
For $n=3$, $k_1=1$ and $k_2=2$, we obtain the generic coadjoint orbit of $\mathsf{SU}(3)$
\begin{align*}
 \mathcal{O}^{\SU(3)}= \mathbb{F}_{1,2}(\mathbb{C}^3) = \{(L,P)\mid L\subset P\subset \mathbb{C}^3 \text{ with } \dim_{\C}(L)=1, \dim_{\C}(P)=2\}.
\end{align*}
This space also appears in the context of so-called {Wallach manifolds} introduced by Wallach \cite{wallach1972compact} which we will describe now. Consider the linear map $J:\mathbb{C}^{2n}\to\mathbb{C}^{2n}$ defined as 
\[ J(z_1,\dots,z_n,z_{n+1},\dots,z_{2n})  = (z_{2n},\dots,z_{n+1},-z_n,\dots, -z_1).\] 
The group $\mathsf{Sp}(n)$ is defined as
\[ \mathsf{Sp}(n) = \{A\in\mathsf{SU}(2n)\mid AJ = J\bar{A}\}. \] 
Let $\mathsf{F}(4)$ be the 52-dimensional exceptional simple Lie group. Moreover, recall that the universal cover of the orthogonal group $\mathsf{SO}(8)$ is called the spin group and is denoted by $\mathsf{Spin}(8)$. 
The Wallach manifolds $W^6$ of dimension six, $W^{12}$ of dimension twelve, and $W^{24}$ of dimension twenty-four are given by
\begin{align*}
W^6:=\frac{\SU(3)}{\U(1)\times\U(1)},\quad 
W^{12}:=\frac{\Sp(3)}{\Sp(1)\times \Sp(1)\times \Sp(1)} \quad\text{and}\quad
W^{24}:=\frac{\mathsf{F}(4)}{\Spin(8)}.
\end{align*}
These are all compact Riemannian manifolds of positive curvature. Moreover, these manifolds can be thought of as the total space of the following homogeneous fibrations:
\begin{align*}
\mathbb{S}^2&\longrightarrow W^6 \longrightarrow \C\mathbb{P}^2,\\
\mathbb{S}^4&\longrightarrow W^{12} \longrightarrow \mathbb{HP}^2,\\
\mathbb{S}^8&\longrightarrow W^{24} \longrightarrow \mathbb{O}\mathbb{P}^2.
\end{align*} 
For more details on these fibrations, we refer the reader to Dearricott $\&$ Galaz-Garci\'a et al.\ \cite{dearricott2014geometry} and the references therein.

\subsection{Bruhat decomposition and induced coordinates} 

So far, we described the coadjoint orbits in terms of matrices and gave a geometrical interpretation in terms of flag manifolds. Now we will focus on the analytical structure which will allow us to determine the Laplace operator to the aim of finding the corresponding Green's function. For that, we need a bit of notation:

%In order to do so, we will consider a complexification of the homogenenous space $\mathsf{SU}(3)/T^2$, which will be more convenient to work with. To that end, we start with the following definition.

A closed subgroup $P$ of a Lie group $G$ is {\em parabolic} if the quotient variety $G/P$ satisfies the following property: for any variety $Y$, the projection map $(G/P)\times Y\to Y$ maps closed set to closed sets. 
	%Varieties satisfying this projection property are called {\em complete}.
	%An important class of examples include projective varieties\footnote{There also exist non-projective complete varieties, see for instance \cite{nagata1958existence}}. 
Furthermore, a closed, connected and solvable subgroup of $G$ is called a {\em Borel subgroup}.
Note that all Borel subgroups are mutually conjugate.

Given a Lie algebra $\mathfrak{g}$ over $\R$, its complexification is defined by $
	\mathfrak{g}^{\C}:=\mathfrak{g}\otimes_{\R}\C.$
%Moreover, recall that a real subalgebra $\mathfrak{f}$ of a complex Lie algebra $\mathfrak{h}$ is called a {\em real form} of $\mathfrak{h}$ if every $h\in\mathfrak{h}$ can be uniquely written as $h=h_1+ih_2$ with $h_1,h_2\in\mathfrak{f}$. The complexification of $\mathfrak{f}$ yields again $\mathfrak{h}$, i.e., $\mathfrak{f}^{\C}\cong \mathfrak{h}$. Note that not every complex Lie algebra has a real form. {\color{green}Moreover, there are in general several non-isomorphic real forms for a given complex Lie algebra: 
%
%\begin{example}
%	The Lie algebra $\mathfrak{sl}(3,\C)$ has the following (non-isomorphic) real forms:
%	\begin{align*}
%		\mathfrak{sl}(3,\R)
%		&=\left\{ \left.
%		\begin{pmatrix}
%			a&b&c\\d&e&f\\g&h&-(a+e)
%		\end{pmatrix}
%		\right| a,b,c,d,e,f,g,h\in \R\right\},
%	\\	\mathfrak{su}(1,2)&=\left\{ \left.
%	\begin{pmatrix}
%			a+bi&c+di&ei\\f+gi&-2bi&-c+di\\hi&-f+gi&-a+bi
%		\end{pmatrix}
%		\right|
%		a,b,c,d,e,f,g,h\in \R\right\},
%	\\	\mathfrak{su}(3)&=\left\{ \left.
%	\begin{pmatrix}
%			ai&c+di&g+hi\\-(c-di)&ib&e+fi\\-(g-hi)&-(e-fi)&-i(a+b)
%		\end{pmatrix}
%		\right| a,b,c,d,e,f,g,h\in \R\right\}.
%	\end{align*}
%	The first one is called the {\em split real form}, the second one {\em quasi-split form}, and the last one is referred to as {\em compact form}.
%\end{example}}
Let $G$ be a compact and connected Lie group. The complexification of $G$ is defined as the complex Lie group $G^{\C}$ that contains $G$ as a closed subgroup and that has the following (universal) property: every homomorphism $f:G\to L$, for every complex Lie group $L$,
lifts to a homomorphism $G^{\C}\to L$. Moreover, on the level of Lie algebras, $\Lie(G^{\C})=\mathfrak{g}^{\C}$ is the complexification of $\Lie(G)=\mathfrak{g}$.

%Now denote by $B$ the subgroup of upper triangular matrices of $\mathsf{SL}(3,\C)$ and by $P$ the subgroup of block upper triangular matrices. $B$ is a Borel subgroup and $P$ a parabolic subgroup.
%\todo{\ssc{\st{verschil `block upper triangular' and `upper triangular'?}}}

The following result was proven in more generality by Picken \cite{picken1990duistermaat}, but we sketch the proof here for the reader's convenience.

\begin{lemma}
\label{lem:isoForFlag}
Denote by $B$ the subgroup of upper triangular matrices of $\mathsf{SL}(3,\C)$. Then there is an isomorphism $\mathsf{SU}(3)/\mathbb{T}^2 \cong \mathsf{SL}(3,\C)/B$.
\end{lemma}

\begin{proof}
%\ggc{$G^c=GB$ and $G\cap B =T$ \[ gT\mapsto gB \]
%Thus \[ [g_c]_B = [gb]_B\mapsto [g]_T \]}
%\todo{\ssc{What is $G$? Is $G^c=GB$ a definition of $G^c$? Is $G\cap B =T$ a definition of $T$? What does the notation with the brackets mean?}(zie onderaan)}
Take a matrix $g\in\mathsf{SL}(3,\C)$ and denote its columns by $\underline{g}_k$ for $k=1,2,3$ so that the matrix can be written as $g=(\underline{g}_1,\underline{g}_2,\underline{g}_3)$. 
Let $\langle \cdot, \cdot \rangle_H$ be the Hermitian inner product on $\C^3$ given by $\langle x, y \rangle_H = x_1 \bar{y}_1 +x_2 \bar{y}_2  +x_3 \bar{y}_3$ for $x=(x_1, x_2, x_3), y=(y_1, y_2, y_3) \in \C^3$.
A priori, the vectors $\underline{g}_k$ are not orthonormal.
Nevertheless, using the Gram-Schmidt procedure they can be made orthonormal:
\begin{align*}
\underline{g}_1'	:&= \underline{g}_1
\\  \underline{g}_2':&=  \underline{g}_2 - \frac{\langle  \underline{g}_2, \underline{g}_1'\rangle}{\langle  \underline{g}_1', \underline{g}_1'\rangle} \underline{g}_1'
\\   \underline{g}_3':&= \underline{g}_3 - \frac{\langle  \underline{g}_3, \underline{g}_2'\rangle}{\langle  \underline{g}_2', \underline{g}_2'\rangle} \underline{g}_2'  - \frac{\langle  \underline{g}_3, \underline{g}_1'\rangle}{\langle  \underline{g}_1', \underline{g}_1'\rangle} \underline{g}_1'
\end{align*}
Normalising each of them via $\underline{v}_k:=\frac{\underline{g}_k'}{||\underline{g}_k'||}$, the matrix given by $U=(\underline{v}_1,\underline{v}_2,\underline{v}_3)$ is an element of $\mathsf{SU}(3)$. Moreover, it satisfies $U=gb'$ for some upper triangular matrix $b' \in B$ which performs the Gram-Schmidt procedure on $g$.
This means that we can write $\mathsf{SL}(3,\C)=\mathsf{SU}(3)B$ and that $\mathsf{SU}(3)\cap B=\mathbb{T}^2$ (this intersection exactly result in the 2-torus). This induces the wanted isomorphism in the following way: take an equivalence class $[g]_B\in\mathsf{SL}(3,\C)/B$ which can also be written as $[Ub]_B$ (see Gram-Schmidt procedure). It is then mapped to $[U]_{\mathbb{T}^2}\in\mathsf{SU}(3)/\mathbb{T}^2$.
%\todo{\ssc{What does the following notation mean? Moreover, why how may and must we divide the $U$ by   induce the wanted isomorphism? Where does the $\mathbb{T}^2$ is divided by? What does the red text at the beginning of the proof mean?}}
%\ggc{In other words, we can map \[ [g]_B=[Ub]_B\mapsto [U]_T. \]}
\end{proof}

The power of this isomorphism lies in the fact that one can make the transition from the geometrical picture of the coadjoint orbit (as being the flag manifold realised as the homogeneous space $\mathsf{SU}(3)/\mathbb{T}^2$) to the complex manifold $\mathsf{SL}(3,\C)/B$. This is convenient as there exists a well-developed theory of so-called Bruhat coordinates on the complex manifold, which will be useful for our approach.

%\begin{theorem}
%	The matrix $A\in\C^{n\times n}$ admits a unique factorization of the form $A=LDU$ where $L$ is a lower triangular matrix having ones on the diagonal, $D$ is a diagonal matrix and $U$ is a upper triangular matrix again with ones on the diagonal if and only if all the matrices $A_{[1,k]}$ for $k=1,\dots,n$ are invertible. 
%\end{theorem}
%More specifically, if we have a matrix $A\in \operatorname{GL}_n(\C)$ then one can always find a permutation matrix $p$ such that $A=pUDL$ can be written down uniquely. \cite{http://nlab-pages.s3.us-east-2.amazonaws.com/nlab/show/Gauss+decomposition}

\begin{definition} Let $\mathfrak{g}$ be a semi-simple Lie algebra with Cartan algebra $\mathfrak{h}$ and root system $\Phi$. The weight space decomposition of a Lie algebra $\mathfrak{g}$ is given by the direct sum decomposition $
		\mathfrak{g} =\mathfrak{h}\oplus  \bigoplus_{\alpha\in\Phi} \mathfrak{g}_{\alpha}
$
where $\mathfrak{g}_{\alpha}=\{X\in\mathfrak{g}\mid [H,X] =\alpha(H) \text{ for } H\in\mathfrak{h}\}$. All $\mathfrak{g}_\alpha$ that are non zero are called {\em roots}.
\end{definition}

Recall that a Lie algebra is {\em simple} if $\mathfrak{g}$ is not abelian and $\mathfrak{g}$ has no non-trivial ideals.	
For a simple Lie algebra $\mathfrak{g}$, we have the triangular decomposition \[ \mathfrak{g} = \mathfrak{n}_- \oplus \mathfrak{h} \oplus \mathfrak{n}_+ \] where $\mathfrak{h}$ is the Cartan subalgebra and $\mathfrak{n}_{\pm}:=\bigoplus_{\alpha\in\Phi^{\pm}}\mathfrak{g}_{\alpha}$ are the so-called upper and lower nilpotent subalgebras consisting of the positive (resp.\ negative) roots of $\mathfrak{g}$.
Moreover, set \[ \mathfrak{b}_{\pm}: = \mathfrak{h} \oplus \mathfrak{n}_{\pm} \] and call them the {\em upper} and {\em lower Borel subalgebras}. On the Lie group level, $B^{\pm}$ and $N^{\pm}$ are called the {\em Borel subgroups} and {\em unipotent subgroups} of the Lie group $G$. In particular, we say that $N^-$ is the {\em opposite} unipotent subgroup.

Let $G$ be a semisimple Lie group with Lie algebra $\mathfrak{g}$. Consider a Borel subgroup $B\leq G$ and the Weyl group $\mathcal{W}$ associated with $G$. Then the {\em Bruhat decomposition} of $G$ is given by $G=\bigcup_{w\in \mathcal{W}}BwB$.
%Denote by $w_0$ the identity element of the Weyl group $W$. Then we obtain $Bw_0B=Uw_0B=w_0U^{-1}B$.
%\todo{\ssc{what is U? Check ook de volgende zinnen...}}
%$U$ is called the {\em big cell} and $U^{-1}$ the {\em opposite big cell}. 
This decomposition gives rise to the {\em cell decomposition} of the homogeneous space $G/B=\bigcup_{w\in \mathcal{W}} BwB/B$.
Each of the $BwB/B$ corresponds to an affine space of dimension $\ell(w)$ where $\ell(w)$ is the length of the Weyl group element $w$ given by the minimal $k$ such that $w$ can be written as a product of $k$ generators of the Weyl group. Note that there is always an element in the Weyl group $\mathcal{W}$ which has maximal length, in this case the length of this element equals the number of positive roots $\Phi^+$. 
This element in the Weyl group is denoted by $w_o\in\mathcal{W}$. It has the property that $w_0(\Phi^+)=\Phi^-$, i.e. it interchanges the positive and negative roots. In the Bruhat decomposition \[ G/B=\bigcup_{w\in \mathcal{W}}BwB/B, \] 
the element $w_0$ gives rise to an open subset $Bw_0B$ which is called the {\em big cell} and is denoted by $X_{w_0}$. The flag manifold is identified with the space $\mathsf{SL}(3,\C)/B$ where $B$ is the subgroup of upper triangular matrices. 
%The {\em opposite} unipotent subgroup is 
In this context, we have
\begin{align*}
N^-= \left\{ \left. 	
\begin{pmatrix}
1	& 0	&	0\\
z_1&1	&	0\\
z_2	&z_3&1
\end{pmatrix}
\right|
z_1,z_2,z_3\in \C
\right\}.
\end{align*} 
\begin{prop}[\cite{boyer1994topology}]
	$N^-$ acts freely and transitively on the big cell $X_{w_0}$. Thus we may identify $X_{w_0}$ with $N^-$.
\end{prop}

Note that the translates $g\cdot N_-$ of the big cell (under the $G$-action) cover the whole flag manifold.

The flag manifold $\mathbb{F}_{1,2}(\C^3)$ can also be seen as
\begin{align}
\mathbb{F}_{1,2}(\C^3)=\{(V_1,V_2)\in \mathbb{CP}^2\times (\mathbb{CP}^2)^*\mid V_2(V_1) =0 \},
\end{align}
where $(\mathbb{CP}^2)^*$ is the {dual complex projective} plane, which means intuitively that the manifold $\mathbb{F}_{1,2}(\C^3)$ consists of all pairs $(V_1,V_2)$ where $V_2$ is a projective line in $\mathbb{CP}^2$ and $V_1$ is a point on the line.

\begin{theorem}[\cite{boyer1994topology}] 
Consider the group $\mathsf{SL}(3,\C)$ and the Borel subgroup $B$ of upper triangular matrices. 
	Let $R_1:=[0:0:1]\in \mathbb{CP}^2$ and $R_2:={\operatorname{span}\{[0:1:0],[0:0:1]\}}\in (\mathbb{CP}^2)^*$. Then the (Bruhat) cell-decomposition of $\mathsf{SL}(3,\C)/B$ consists of the following six cells (where the indices in the Bruhat cell correspond to the group elements of the symmetric group in three elements, see Table \ref{table:bruhatExpression}):
	\begin{enumerate} 
		\item 
		The {\em big cell}, which has codimension zero, is given by 
		\begin{align*}
			X_e := \{(V_1,V_2)\in \mathbb{F}_{1,2}(\C^3)\mid  R_2 (V_1) \neq 0, \  V_2 (R_1) \neq 0 \}.
		\end{align*}
		\item 
		There are two Bruhat cells of codimension one given by 
		\begin{align*}
			X_{(1,2)} &:= \{(V_1,V_2)\in \mathbb{F}_{1,2}(\C^3)\mid  R_2 (V_1)=0, \ R_ 1\neq V_1, \  R_2\neq V_2\}, \\
			X_{(2,3)} &: = \{(V_1,V_2)\in \mathbb{F}_{1,2}(\C^3)\mid  V_2 (R_1)=0, \   R_ 1\neq V_1, \  R_2\neq V_2  \} .
			\end{align*}
		\item 
		There are two Bruhat cells of codimension two given by 
		\begin{align*}
			X_{(1,2,3)} &:= \{(V_1,V_2)\in \mathbb{F}_{1,2}(\C^3)\mid V_1= R_1, \ R_2\neq V_2\} , 
			\\X_{(1,3,2)} &:= \{(V_1,V_2)\in \mathbb{F}_{1,2}(\C^3)\mid V_2 = R_2, \ R_1\neq V_1\} .
			\end{align*}
		\item 
		The $0$-cell, which has codimension three, is given by 
		 \begin{align*}
			X_{(1,3)} := \{(V_1,V_2)\in \mathbb{F}_{1,2}(\C^3)\mid V_1= R_1, \ R_2= V_2\}.
		\end{align*}
	\end{enumerate}
	Moreover, the opposite unipotent subgroup $N^-$ acts transitively on the big cell $X$. 
\end{theorem}
Recall that the Weyl group of $\mathsf{SL}(3,\C)$ is isomorphic to the {symmetric group  $\mathsf{Sym}(3)$ which has order $|\mathsf{Sym}(3)|=6$}. We now give an overview how the elements of $\mathsf{Sym}(3) $ correspond to the elements of the Weyl group and their associated Bruhat decomposition and Weyl length.

\begin{table}
\begin{center}
\begin{tabular}{|c|c|c|c|}
	\hline  
	\textsc{Group element} & \textsc{Bruhat expression} & \textsc{Length} & \textsc{Matrix representation} \\ 
	\hline \hline
	$e$ & empty word & 0 & $\begin{pmatrix} 1&0&0\\0&1&0\\0&0&1 \end{pmatrix}$ \\ 
	\hline 
	$(1,2)$ & $s_1$ & 1 & $\begin{pmatrix} 0&-1&0\\1&0&0\\0&0&1 \end{pmatrix}$ \\ 
	\hline 
	$(2,3)$ & $s_2$ & 1 & $\begin{pmatrix} 1&0&0\\0&0&1\\0&-1&0 \end{pmatrix}$  \\ 
	\hline 
	$(1,2,3)$ & $s_1s_2$ & 2 & $\begin{pmatrix} 0&0&-1\\1&0&0\\0&-1&0 \end{pmatrix}$   \\ 
	\hline 
	$(1,3,2)$ & $s_2s_1$ &  2 & $\begin{pmatrix}
	0&-1&0\\0&0&1\\-1&0&0
	\end{pmatrix}$\\ 
	\hline 
	$(1,3)$ & $s_1s_2s_1$ & 3 & $\begin{pmatrix}
	0&0&-1\\0&-1&0\\-1 &0&0
	\end{pmatrix}$ \\ 
	\hline 
\end{tabular} 
\end{center}
\caption{Bruhat expressions.}
\label{table:bruhatExpression}
\end{table}

To work on the flag manifold $\mathbb{F}_{1,2}(\C^3)$, we need explicit coordinates.  

\begin{corollary}\label{localco}
The big cell of the flag manifold $\mathbb{F}_{1,2}(\C^3) = \mathsf{SL}(3,\C) \slash B$ can be identified with $N^-$. This leads to the following coordinate chart for the big cell of $\mathbb{F}_{1,2}(\C^3)$:
\begin{align}
\label{coordinates}
N^- \to \C^3 \simeq \R^6, 
\quad 
\begin{pmatrix}
1	& 0	&	0\\
z_1&1	&	0\\
z_2	&z_3&1
\end{pmatrix}
\mapsto(z_1,z_2,z_3) \simeq (x_1,x_2,x_3,y_1,y_2,y_3).
\end{align}
\end{corollary}

%%%%%%%%%%%%%%%%%%%%%%%%%%%%%%%%%%%%%%%%%%%%
%%%%%%%%%%%%%%%%% new subsection %%%%%%%%%%%%%

\subsection{The K\"ahler structure of coadjoint orbits} 

Let $M$ be a complex manifold of complex dimension $n$ with local coordinates $(z_1,\dots,z_n)\in\C^n$. Then a {\em Hermitian metric} $h$ is of the form 
\begin{align*}
		h = \sum_{i,j=1}^n h_{ij}dz_i\otimes d\overline{z}_j
\end{align*}
where $(h_{ij})_{1 \leq i, j \leq n}$ is a positive-definite Hermitian matrix. 
A complex manifold $M$ equipped with a Hermitian metric $h$ is called a {\em Hermitian manifold}.
A Hermitian manifold $(M, h)$ carries a natural symplectic form, more precisely,  
the $(1,1)$-form given by imaginary part of the Hermitian metric $h$ is symplectic and has the explicit expression 
\begin{align*}
\omega & :=-\Im(h)=-\frac{1}{2i}(h-\overline{h})  = \frac{i}{2}\sum_{1\leq i,j\leq n}h_{ij}dz_i\otimes d\overline{z}_j-h_{ji}d\overline{z}_i\otimes{dz}_j\\
&=\frac{i}{2}\sum_{1\leq i,j\leq n}h_{ij}(dz_i\otimes d\overline{z}_j-d\overline{z}_j\otimes dz_i)
=\frac{i}{2}\sum_{1\leq i,j\leq n}h_{ij}dz_i\wedge d\overline{z}_j.
\end{align*}
This $\omega$ is often referred to as the \textit{fundamental form} on $(M, h)$.

An almost complex structure $J$ on a smooth manifold $M$ is an isomorphism $J: TM \to TM$ with $J^2=-\operatorname{Id}$. Such a $J$ is integrable if the so-called {\em Nijenhuis tensor} \[ N_J(X,Y):= [X,Y] + J[JX,Y] + J[X,JY]- [JX,JY]\] vanishes for all vector fields $X,Y$ on the manifold $M$.

A symplectic manifold $(M, \omega)$ is {\em K\"ahler} if there exists an integrable almost complex structure $J$ such that the bilinear form $g(u,v):=\omega(u,Jv)$ is symmetric and positive definite for all $u,v\in TM$, i.e., $g$ is a Riemannian metric.

A Hermitian manifold $(M, h)$ resp.\ $h$ is {\em K\"ahler} if the fundamental form $-\Im(h)$ is closed, i.e.\ $-d\Im(h)=0$. Moreover, in this situation, $-\Im(h)$ is in fact a (real) symplectic form on $(M, h)$.

\begin{lemma}[Hou \& Hou \cite{hou1997differential}]
Let $(M, h)$ be a  K\"ahler manifold. Then for all $p \in M$ there exists a open neighbourhood $U$ of $p$ and a function $K_U: U \to \R$ such that
\[ h_{ij} = \partial_{z_i}\partial_{\bar{z_j}} K_U(z_i,\bar{z_j})  \quad \mbox{for all } 1 \leq i, j \leq n \]
for local complex coordinates $z$ on $U$. This locally defined function is usually called the \textit{K\"ahler potential} and denoted by $K_M$.
\end{lemma}
%\todo{\ssc{the above lemma was originally a definition, but I think that it is rather a lemma that includes a notation: Why does this $U$ and $K_U$ actually exist??}}

Note that $(h_{ij})_{1 \leq i, j \leq n}$ is defined globally, whereas the potential is only defined locally. Using the Dolbeault operators
\[\partial : =\sum_{k=1}^n\partial_{z_k}dz_k
\qquad\text{and}\qquad 
\overline{\partial}:=\sum_{k=1}^n\partial_{\overline{z}_k}d\overline{z}_k,\]
the fundamental $(1, 1)$-form can be expressed as $\omega=i\partial\overline{\partial}K_M$.

\begin{example}
On $\R^{2n}\cong \C^n$, consider the {Euclidean metric $g_E$}, the standard symplectic form $\omega_{st}$, and standard compatible complex structure $J_{st}$ given in matrix notation by
\begin{align*}
	g_E = \begin{pmatrix}
		I_n	&	0
		\\ 0 &	I_n
	\end{pmatrix}
	,&&
	\omega_{st} = \begin{pmatrix}
		0	&	I_n
		\\ -I_n&	0
	\end{pmatrix}
		,&& 
		J_{st}= \begin{pmatrix}
		0	&	-I_n
		\\ I_n&	0
	\end{pmatrix}
\end{align*}
where $I_n$ is the $(n \times n)$-unit matrix.
Then $K_{\C^n}: \C^n \to \R$ given by $K_{\C^n}(z):=\frac{|z|}{2}$ is a K\"ahler potential since
\begin{align*}
	 i\partial\overline{\partial}\left(\frac{|z|}{2}\right) =\frac{i}{2}\partial\overline{\partial}\sum_{k=1}^nz_k\overline{z}_k=\frac{i}{2}\sum_{k=1}^ndz_k\wedge d\overline{z}_k.
\end{align*}
\end{example}

An important class of K\"ahler manifolds is given by coadjoint orbits:

\begin{theorem}[Borel, \cite{Borel1954}]
Let $G$ be a semi-simple compact Lie group. Each (co)adjoint orbit has a  $G$-invariant K\"ahler structure.
\end{theorem}

%\[ \omega=\frac{i}{2}\partial\overline{\partial}\ln(1+|z|^2). \] The function $f=\ln(1+|z|^2)$ is  the {K\"ahler potential} in this case. The action of the group $U(V)$ on $V=\C^N$ and on $\mathbb{P}(V)$. Take an element in the Lie algebra $X\in\mathfrak{u}(V)$, this gives rise to a vector field $X^V(v)=-X\cdot V$. The momentum map is given by  \begin{align*}
%\mu_V:V\to\mathfrak{u}(V) ^*\\ \langle \mu_V(v),X\rangle = -\frac{i}{2}(X\cdot v,v)
%\end{align*}
%It follows that the momentum map 
%\begin{align*}
%\mu_V:\mathbb{P}V\to\mathfrak{u}(V) ^*\\ \langle \mu_V(v),X\rangle = -\frac{i}{2}\frac{(X\cdot v,v)}{||v||^2}
%\end{align*}
%In the case where $N=1$ we have \[ \omega=\frac{dx\wedge dy}{((|z|^2+1)^2} ,\] the momentum map is \[ \mu(z)=\frac{1}{2}\frac{|z|^2}{|z|^2+1}. \]
%{\color{red}\texttt{Write this properly down for general projective space (in homogeneous coordinates suddenly?)}}

The idea is to study the point vortex dynamics modelled on the degenerate orbit $\mathbb{CP}^2$ and the generic orbit given by flag manifold ${F}_{1,2}(\mathbb{C}^3)$. Therefore it is useful to know their K\"ahler potentials.

\begin{lemma}[{Picken \cite{picken1990duistermaat}}]
\label{lem:potentialOnOrbits}
The K\"ahler potentials on $\mathbb{CP}^n$ and on the flag manifold are given by the following logarithmic functions depending on local coordinates $(1:z_1:z_2:\cdots:z_n)$ on $\mathbb{CP}^n$ and the local coordinates from Corollary \ref{localco} for the flag manifold $\mathbb{F}_{1,2}(\mathbb{C}^3)$.
\begin{align}\label{potential}
K_{\mathbb{CP}^n}&=  \log\left(1+\sum_{k=1}^{n}|z_k|^2\right),
\\ K_{{\mathbb{F}_{1,2}(\mathbb{C}^3)}}&=  \log\left( (1+|z_1|^2+|z_2|^2)(1+|z_3|^2+|z_1z_3-z_2|^2)	\right) =: \log(K_1 K_2).
\end{align}
\end{lemma}

Moreover,

\begin{lemma}[{Mu\~noz $\&$ Gonz\'alez-Prieto $\&$ Rojo \cite{munoz2020geometry}}]
In homogeneous coordinates $(1:z_1:z_2:\cdots:z_n)\in\mathbb{CP}^n$, the Hermitian metric $h= (h_{ij})_{1 \leq i, j \leq n}$ on $\mathbb{CP}^n$ takes the following form:
 $$
 (h_{ij})_{1 \leq i,j \leq n}={\frac {1}{(1+|{z} |^{2})^{2}}}
{\begin{pmatrix}1+|{z} |^{2}-|z_{1}|^{2}&-{\bar {z}}_{1}z_{2}&\cdots &-{\bar {z}}_{1}z_{n}\\-{\bar {z}}_{2}z_{1}&1+|{z} |^{2}-|z_{2}|^{2}&\cdots &-{\bar {z}}_{2}z_{n}\\\vdots &\vdots &\ddots &\vdots \\-{\bar {z}}_{n}z_{1}&-{\bar {z}}_{n}z_{2}&\cdots &1+|{z} |^{2}-|z_{n}|^{2}\end{pmatrix}}.
$$
Its determinant is given by $\det (h_{ij}) = \frac{1}{(1+|z|^2)^{n+1}}$.
\end{lemma}

\begin{lemma} 
\label{hermitianmetric}
Recall from Lemma \ref{lem:potentialOnOrbits} the real valued functions
\begin{align*}
	K_1  = 1 + |z_1|^2 + |z_2|^2 
	\quad \mbox{and} \quad 
	K_2  = 1 + |z_3|^2 + |z_1z_3-z_2|^2.
\end{align*}
Then the Hermitian metric $(h_{ij})_{1\leq i\leq j\leq 3}$ on $\mathbb{F}_{1,2}(\mathbb{C}^3)$ has the following matrix representation:
	$$(h_{ij})_{1 \leq i,j \leq n} =	\begin{pmatrix}
	\frac{1+|z_2|^2}{K_1^2} + \frac{|z_3|^2(1+|z_3|^2)}{K_2^2}	&	-\frac{\overline{z}_1z_2}{K_1^2}-\frac{z_3(1+|z_3|^2)}{K_2^2}	&	\frac{z_3(\overline{z}_1+\overline{z}_2z_3)}{K_2^2}\\
	-\frac{{z}_1\overline{z}_2}{K_1^2}-\frac{\overline{z}_3(1+|{z}_3|^2)}{K_2^2}	&	\frac{1+|z_1|^2}{K_1^2} + \frac{1+|z_3|^2}{K_2^2}	&-	\frac{\overline{z}_1+\overline{z}_2z_3}{K_2^2}	\\
	\frac{{\overline{z}}_3({z}_1+{z}_2\overline{z}_3)}{K_2^2}	&	-\frac{{z}_1+{z}_2\overline{z}_3}{K_2^2}	&	\frac{K_1}{K_2^2}
	\end{pmatrix}.
	$$
	Its determinant is given by $\det \bigl( (h_{ij})_{1 \leq i,j \leq n} \bigr) = \frac{2}{K_1^2K_2^2}$.
\end{lemma}

\begin{proof} 
Recall from Lemma \ref{lem:potentialOnOrbits} the expression for the K\"ahler potential 
\begin{align*}
	\log (1+|z_1|^2+|z_2|^2)(1+|z_3|^2+|z_1z_3-z_2|^2) = \log K_1K_2 = \log K_1+\log K_2.
	\end{align*}
The entries of the matrix $(h_{ij})_{1\leq i\leq j\leq 3}$ are computed via $h_{ij} = \partial_{z_i}\partial_{\bar{z_j}} K_M(z_i,\bar{z_j})$. Exemplarily we now compute the entry
\[ h_{11} = \partial_{z_1}\partial_{\overline{z}_1} \log((1+|z_1|^2+|z_2|^2)) +  \partial_{z_1}\partial_{\overline{z}_1} \log(1+|z_3|^2+|z_1z_3-z_2|^2). \]
	The first term in this expression becomes 
	\begin{align*}
	\partial_{z_1}\frac{z_1}{1+|z_1|^2+|z_2|^2} = \frac{(1+|z_1|^2+|z_2|^2)-z_1\overline{z}_1}{(1+|z_1|^2+|z_2|^2)^2}=\frac{1+|z_2|^2}{K_1^2}
	\end{align*}
	and the second one  
	\begin{align*}
	\partial_{z_1} \frac{\overline{z}_3(z_1z_3-z_2)}{1+|z_3|^2+|z_1z_3-z_2|^2} & = \frac{|z_3|^2(1+|z_3|^2+|z_1z_3-z_2|^2)-|z_3|^2|z_1z_3-z_2|^2}{(1+|z_3|^2+|z_1z_3-z_2|^2)^2}  \\
	& = \frac{|z_3|^2(1+|z_3|^2)}{K_2^2}.
	\end{align*}
	Altogether, we obtain
	\[ h_{11} = \frac{1+|z_2|^2}{K_1^2} + \frac{|z_3|^2(1+|z_3|^2)}{K_2^2}. \]
	The other entries are computed analogously. 
\end{proof}

A straightforward computation yields:

\begin{corollary}
	The inverse matrix $((h_{ij})_{1\leq i\leq j\leq 3})^{-1}=: (h^{ij})_{1\leq i\leq j\leq 3}$ is given by: 
	\begin{align*}
	{\footnotesize\begin{pmatrix}
	K_1\left(1+|z_1|^2+\frac{K_1}{K_2}\right) &
	K_1\left(\overline{z}_1z_2+\frac{K_1}{K_2}{z}_3\right)&
	(\overline{z}_1+z_3\overline{z}_2)(\overline{z}_1z_2-z_3-z_3|z_1|^2)\\
	K_1\left(z_1\overline{z}_2+\frac{K_1}{K_2}\overline{z}_3\right)& 
	K_1\left((1+|z_2|^2)+\frac{K_1}{K_2}|z_3|^2\right)&
	(\overline{z}_1+\overline{z}_2z_3)\left((1+|z_2|^2)-z_1\overline{z}_2z_3\right)\\
	(z_1+\overline{z}_3z_2)(z_1\overline{z}_2-\overline{z}_3-z_3|z_1|^2)
	&
	({z}_1+{z}_2\overline{z}_3)\left((1+|z_2|^2)-{\overline{z}}_1{z}_2\overline{z}_3\right)&
	K_1(1+|z_3|^2)+\frac{K_2^2}{K_1}
	\end{pmatrix}}
	\end{align*}
	where $K_1  = 1 + |z_1|^2 + |z_2|^2$  and $K_2  = 1 + |z_3|^2 + |z_1z_3-z_2|^2$.
\end{corollary}

%%%%%%%%%%%%%%%%%%%%%%%%%%%%%%%%%%%%%%%%
%%%%%%%%%%%%%%% new subsection %%%%%%%%%%%

\subsection{Different symplectic structures} 
Important for us is the following result due to Kirillov, Kostant and Souriau:

\begin{theorem}
Let $G$ be a Lie group and $\mathfrak g$ its Lie algebra with dual $\mathfrak{g}^*$ and $\mu\in\mathfrak{g}^*$.
Then the coadjoint orbit $\mathcal{O}_{\mu}$ carries the canonical symplectic form
\begin{align*}
	\omega^{KKS}_\nu (\operatorname{ad}_{\xi}^*\nu,\operatorname{ad}_{\eta}^*\nu): =\langle\nu,[\xi,\eta] \rangle,
\end{align*}	
where $\xi,\eta\in\mathfrak{g}$ and $\nu\in\mathcal{O}_{\mu}$. This symplectic form is usually called the {\em Kirillov-Kostant-Souriau (KKS) symplectic form}.
\end{theorem}
%However, as we just saw, from an analytical viewpoint in can be more convening to work with its complexification $M^{\mathbb{C}}=\mathsf{SL}(3)/B$. Recall that the Kirillov symplectic form takes two tangent vectors as its input. So in order to obtain this symplectic form we must know the tangent space of the underlying manifold.  
%
%\begin{lemma}
%	The tangent space at $x$ to an orbit $\mathcal{O}_{x_0}$ is given by \[ T_x\mathcal{O}_{x_0}=\{\xi_M(x)\mid \xi \in \mathfrak{g}\}. \] Furthermore, for the adjoint action the infinitesimal generators are given by $\xi_{\mathfrak{g}}(\eta)=[\xi,\eta]$.
%\end{lemma}
%\begin{enumerate}
%	\item We first observe that every flag manifold can be realised as a homogenous space $G/P$ and that the tangent space is then given by the corresponding algebra $\mathfrak{g}/\mathfrak{p}$.
%	\item The following observation is not general for all flag manifold, but it works for the generic coadjoint orbit of $\SU(3)$. Take a point $B\in \mathcal{O}_{\lambda}$ then we can write by  $B=A\Lambda A^{-1}$ such that the tangent space is given by all the commutators with respect to and element $C$ in the Lie algebra $\mathfrak{su}(3)$. In other words: \[ T_B\mathcal{O}_{\lambda} =  \{ [C,B]\mid C\in\mathfrak{su}(n)	\}.\]
%	Now take two element in this tangent space $[C,B]$ and $[D,B]$ then the Kirillov symplectic form  takes the following form
%	\begin{align*}
%	\omega_B ([C,B],[D,B]):&=\mathsf{trace}(B^{\dagger}[C,D])
%	\\&\omega_B(C,D)=\trace([C,D]B)
%	\end{align*}
%\end{enumerate}

This implies that, considered as coadjoint orbit, the flag manifold $\mathbb{F}_{1,2}(\mathbb{C}^3) = \mathsf{SU}(3)/\mathbb T^2$ can be endowed with $\omega^{KKS}$ as symplectic form. Note that there is an additional way to consider $\mathbb{F}_{1,2}(\mathbb{C}^3)$ as symplectic manifold: We consider the complexification of $\mathsf{SU}(3)/\mathbb T^2$ given by $\left( \mathsf{SU}(3)/\mathbb T^2\right) ^{\mathbb{C}}=\mathsf{SL}(3, \C)/B$. Using the coordinates on the big cell given in \eqref{coordinates} and the Hermitian metric from Lemma \ref{hermitianmetric}, Picken $\&$ Duistermaat \cite{picken1990duistermaat} give the following formula for a  symplectic form on $(W^6)^\C$:
\begin{align*}
\omega^{(W^6)^\C}:=\frac{i}{2}\left(	\partial\overline{\partial}\log\left(1+\sum_{k=1}^2|z_k|^2\right) + \partial\overline{\partial}\log\left(1+|z_3|^2+|z_1z_3-z_2|^2\right)	\right).
\end{align*}
Note that, if $\omega^{\mathbb{CP}^2}$ denotes the Fubini-Study form on $\mathbb{CP}^2$ then the symplectic form $\omega^{(W^6)^\C}$ consists of the Fubini-Study form $\omega^{\mathbb{CP}^2} = \frac{i}{2}	\partial\overline{\partial}\log\left(1+\sum_{k=1}^2|z_k|^2\right)$ on $\mathbb{CP}^2$ plus the correction term
$
\tilde{\omega} = \frac{i}{2} \partial\overline{\partial}\log\left(1+|z_3|^2+|z_1z_3-z_2|^2\right),
$
i.e., $\omega^{(W^6)^\C}=\omega^{\mathbb{CP}^2}+\tilde{\omega}$.

\begin{remark}[{Picken \cite{picken1990duistermaat}, Bernatska $\&$ Holod et al.\ \cite{bernatska2008geometry}}]
  $
 \left( (W^6)^\C, \ \omega^{(W^6)^\C} \right) 
 $
 and
 $
 \left( \mathbb{F}_{1,2}(\mathbb{C}^3) = \mathsf{SU}(3)/\mathbb T^2, \ \omega^{KKS} \right)
 $
 are symplectomorphic.
\end{remark}

Summarized, we have the following types of coadjoint orbits of $\mathsf{SU}(3)$, each endowed with its natural symplectic structure:

\begin{table}[ht]
	\centering
	\setstretch{1.5}
	\begin{tabular}{|c |c| c|}
		\hline
		 {coadjoint orbit} &  {symplectic form} & {(real) dimension}  \\ 
		\hline \hline
		$\frac{\SU(3)}{\U(1)\times\U(1)}$ & Kirillov-Kostant-Souriau &  6\\ \hline
		
		$\mathbb{CP}^2$ & Fubini-Study & 4 \\ \hline
		
		point & trivial & 0 \\ 
		\hline 
	\end{tabular} 
	\setstretch{1.0}
		\caption{Coadjoint orbits of $\SU(3)$.}

\end{table}

%%%%%%%%%%%%%%%%%%%%%%%%%%%%%%%%%%%%%%%%%%%%%%%%%
%%%%%%%%%%%%%% new section %%%%%%%%%%%%%%%%%%%%%%
%%%%%%%%%%%%%%%%%%%%%%%%%%%%%%%%%%%%%%%%%%%%%%%%%%

\section{The point vortex momentum map on $\mathbb{CP}^2$ and $\mathbb{F}_{1, 2}(\C^3)$}
\label{sectionMomentumMap}
In this section, we will study the Hamiltonian action of $\mathsf{SU}(3)$ on (products of) coadjoint orbits. In the case of the degenerate orbit, the dynamics have been studied before: for example, the Hamiltonian action of $\mathsf{SU}(3)$ on $\mathbb{CP}^2\times \mathbb{CP}^2$ has been studied by Beddulli $\&$ Gori \cite{bedulli2007deformations}. Moreover, Montaldi $\&$ Shaddad \cite{montaldi2018generalized} considered a similar problem but added a copy of the projective plane. To be more precise, they considered the (diagonal) action of $\mathsf{SU}(3)$ on $\mathbb{CP}^2\times \mathbb{CP}^2\times \mathbb{CP}^2$ and the associated properties of the (weighted) momentum map. We will focus on the generic orbit, which is the six-dimensional flag manifold and construct a momentum map $\mu:\mathcal{O}^{\mathsf{SU}(3)}\to\mathfrak{su}(3)^*$ explicitly.

%%%%%%%%%%%%%%%%%%%%%%%%%%%%%%%%%%%%%%%%
%%%%%%%%%%% new subsection  %%%%%%%%%%%%

\subsection{The momentum map for vortex dynamics} 
 %Let $(M_1,\omega_1)$ and $(M_2,\omega_2)$ be symplectic manifolds and consider a Hamiltonian group action of $G$ on $M_1$ and $M_2$. We denote the corresponding momentum maps by $\mu_1:M\to\mathfrak{g}^*$ and $\mu_2:M\to\mathfrak{g}^*$. We now consider the product manifold $(M_1\times M_2,\omega)$ where $\omega$ is the product symplectic form induced by $\omega_1$ and $\omega_2$. We now consider the diagonal action of $G$ on $M_1\times M_2$, i.e. $$g\cdot (m_1,m_2)=(g\cdot m_1,g\cdot m_2).$$ This action is also Hamiltonian and has momentum map $$\mu_1\times \mu_2:M_1\times M_2\to \mathfrak{g}^*, \quad (m_1,m_2)\mapsto \mu(m_1)+\mu(m_2).$$ 

Let $N \in \mathbb N$ and, for $1 \leq k \leq N$, let $\Gamma_k \in \R^{\neq 0}$ (`weight') and let $(M_k,\omega_k)$ be a symplectic manifold. Let $G$ be a Lie group that acts on each $(M_k, \omega_k)$ with momentum map $\mu_k: M_k \to \mathfrak{g}^*$. Now set $M:= \Pi_{k=1}^N M_k$ and equip it with the weighted symplectic form $\omega_M:= \sum_{k=1}^N\Gamma_k	\tau_k^*\omega_{k}$ where $\tau_k : M\to M_k$ is the projection on the $k$th factor. The diagonal action of $G$ on $M$ is given by $g.m:=(g.m_1, \dots, g. m_N)$ for $g \in G$ and $m = (m_1, \dots, m_N) \in M$ and its momentum map is given by
\begin{align*}
	\mu_M:M\to \mathfrak{g}^*, \qquad \mu_M(m_1,\dots,m_N) = \sum_{k=1}^N\Gamma_k\mu_k(m_k).
\end{align*}
We are interested in the special situation when the symplectic manifolds $M_k$ are coadjoint orbits i.e.\ $(M_k, \omega_k) = (\mathcal{O},\omega_{\mathcal{O}})$.
%
%Let $G$ be a Lie group acting on $\mathfrak{g}^*$ via its coadjoint action and denote by $(\mathcal{O},\omega_{\mathcal{O}})$ one of its coadjoint orbits equipped with its symplectic form. Moreover, denote by $M := \prod_{k=1}^N\mathcal{O}_k$ the space consisting of the product of $N$ coadjoint orbits $\mathcal{O}_k$ (which is often referred to as `phase space'). We further denote by 
%$$\tau_k:M\to \mathcal{O}_k,\quad (t_1,\dots,t_N)\mapsto t_k$$ 
%the canonical projection on the $k$th component of the product and by $\tau_k^*$ the pullback of this projection. Let $\Gamma_1,\dots,\Gamma_N \in \R^{\neq 0}$.   
%These constants are usually called \textit{weights} in the context of vortex dynamics. Then the manifold $M$ is symplectic with respect to the `weighted' symplectic form 
%\begin{align*}
%	\omega_M:= \sum_{k=1}^N\Gamma_k	\tau_k^*\omega_{\mathcal{O}_k}.
%\end{align*}
%
%
In the next subsections, we study momentum maps of vortex dynamics for the following two situations:
\begin{align*}
	\mu_ {\mathcal{O}_d^{\mathsf{SU}(3)}}: \mathcal{O}_d^{\mathsf{SU}(3)} \simeq \mathbb{CP}^2 \to \mathfrak{su}(3)^*
	\quad\text{and}\quad 	
	\mu_ {\mathcal{O}^{\mathsf{SU}(3)} }: \mathcal{O}^{\mathsf{SU}(3)} \simeq \mathbb{F}_{1,2} (\C^3)  \to \mathfrak{su}(3)^*.
\end{align*}
Recall that we identify $\mathfrak{su}(3)$ with $\mathfrak{su}(3)^*$ using the Killing form. Thus, $\mathfrak{su}(3)\cong \mathfrak{su}(3)^*$ is identified with the space of complex skew-Hermitian matrices with trace zero.

\subsection{The momentum map of the degenerate orbit $\mathcal{O}_d^{\mathsf{SU}(3)} \simeq \mathbb{CP}^2$}

In this subsection, we recall some facts from Montaldi $\&$ Shaddad \cite{Montaldi2019} concerning the momentum map of the degenerate coadjoint orbit $\mathcal{O}_d^{\mathsf{SU}(3)} \simeq \mathbb{CP}^2$ of $\mathsf{SU}(3)$.

%
%\begin{notation}
%	Let $A$ be an $(m\times n)$-matrix and $B$ be a $(p\times q)$-matrix over a field $\mathbb{K}$.  The \textit{Kronecker product} $A\otimes B$ is the $(mp \times nq)$-matrix defined by \begin{align*}
%	A\otimes B := 
%	\begin{pmatrix}
%		a_{11}B	&	\cdots	&	a_{1n}B\\
%		\vdots	&	\ddots	&	\vdots\\
%		a_{m1}B	&	\cdots	&	a_{mn}B
%	\end{pmatrix}. 
%\end{align*}
%\end{notation}

\begin{theorem}[{Montaldi \& Shaddad \cite{Montaldi2019}}]
The momentum map for the Fubini-Study form on $\mathbb{CP}^2$ is given by  
\begin{align*}
	\mu:\mathbb{CP}^2\to\mathfrak{su}(3)^*,\quad [x:y:z]\mapsto  {i}\begin{pmatrix}
	|x|^2-\frac{1}{3}	&	x\overline{y}	&	x\overline{z}	\\
	\overline{x}y			&	|y|^2-\frac{1}{3}	&	y\overline{z}	\\
	\overline{x}z		&	\overline{y}z	&	|z|^2-\frac{1}{3}
	\end{pmatrix}.
	\end{align*}
	Furthermore, the map satisfies the following properties:
	\begin{enumerate}[\normalfont(i)]
		\item 
		$\mu$ is $\mathsf{SU}(3)$-equivariant for the left action, i.e. $\mu(g Z)=g \mu(Z)$ for all $g\in\SU(3)$ and all $Z\in \mathbb{CP}^2$.
		\item 
		The image of $\mu$ consists of $(3\times 3)$ Hermitian matrices with eigenvalues $-\frac{1}{3}, -\frac{1}{3}$ and $\frac{2}{3}$.
	\end{enumerate}
\end{theorem}

\begin{proof}
We briefly sketch a part of the proof: the characteristic polynomial $\chi(u)$ of the matrix $\mu(x,y,z)$ is given by
	$$	
	\chi(u) = (|x|^2+|y|^2+|z|^2)\left(u^2  + \frac{2}{3}u+\frac{1}{9}\right)-u^3-u^2-\frac{u}{3}-\frac{1}{27}.
	$$
	Solving the equation $\chi(u)=0$ and using the fact that $|x|^2+|y|^2+|z|^2=1$ gives the three eigenvalues $u_1	=-\frac{1}{3}$, $u_2=-\frac{1}{3}$ and $u_3=\frac{2}{3}$.
\end{proof}

\subsection{The momentum map of the generic orbit $  \mathcal{O}^{\mathsf{SU}(3)} \simeq \mathbb{F}_{1,2} (\C^3) $}
\label{momentflag}

In order to obtain the momentum map
on the flag manifold $\mathbb{F}_{1,2} (\C^3)\simeq \mathcal{O}^{\mathsf{SU}(3)}$ we need to have the infinitesimal generators of the Lie algebra $\mathfrak{su}(3)$ at our disposal. They are provided by the following statement:

\begin{lemma}
\label{lem:infinitesimalVF}
Let $\lambda_1, \dots, \lambda_8$ be the rescaled basis from Notation \ref{rescaledGel}. Then the infinitesimal vector fields of the Lie algebra $\mathfrak{su}(3)$ on the flag manifold are given by 
\begin{align*}
	X_{\lambda_1}&=\frac{i}{2}\left((1-z_1^2)\partial_{z_1}-z_1z_2\partial_{z_2}+(z_1z_3-z_2)\partial_{z_3}\right),\\
	X_{\lambda_2}&=\frac{1}{2}\left((-z_1^2-1)\partial_{z_1}-z_1z_2\partial_{z_1}+(z_1z_3-z_2)\partial_{z_3}\right),
\\	X_{\lambda_3}&=\frac{i}{2}  \left(-2 z_1\partial_{z_1}-z_2\partial_{z_2}+z_3\partial_{z_3}\right),
\\	X_{\lambda_4}&=\frac{i}{2}\left(-z_1z_2\partial_{z_1}+(1-z_2^2)\partial_{z_2}-z_3(z_2-z_1z_3)\partial_{z_3}\right),
\\ X_{\lambda_{5}}&=\frac{1}{2}\left(-z_1z_2\partial_{z_1}-(1+z_2^2)\partial_{z_2}-z_3(z_2-z_1z_3)\partial_{z_3}\right),
\\ X_{\lambda_6}&=\frac{i}{2}\left(z_2\partial_{z_1}+z_1\partial_{z_2}+(1-z_3^2)\partial_{z_3}\right),
\\ X_{\lambda_7}&=\frac{1}{2}\left(z_2\partial_{z_1}-z_1\partial_{z_2}-(1-z_3^2)\partial_{z_3}\right),
\\ X_{\lambda_8}&=-\frac{i\sqrt{3}}{2}\left(z_2\partial_{z_2}+z_3\partial_{z_3}\right).
\end{align*} 
\end{lemma}

\begin{proof} 
In order to obtain the fundamental vector fields associated to $\mathfrak{su}(3)$ it is sufficient to determine the vector fields associated to the basis $\lambda_1, \dots, \lambda_8$ from Notation \ref{rescaledGel}. 
The vector fields are determined by the equation	
\begin{align*}
		X_{\lambda_k}=\frac{d}{dt} {\bigg|_{t=0}} \exp(t \lambda_k)\cdot \mathcal{Z}
\end{align*}
where $\mathcal{Z}\in \mathbb{F}_{1,2} (\C^3) \simeq \mathcal{O}^{\mathsf{SU}(3)}$ (see local coordinates expression from Corollary \ref{localco})
and $ \exp(t \lambda_k) \cdot \mathcal{Z}$ is defined as multiplication of the matrices $\exp(t \lambda_k)$ and $ \mathcal{Z}$ and corresponds to the left action of $\mathsf{SU}(3)$ on the flag manifold.

Without loss of generality, $\mathcal{Z} $ may lie in the big cell and thus is of the form $ \mathcal{Z} = \left( 
\begin{smallmatrix}
1	& 0	&	0\\
z_1&1	&	0\\
z_2	&z_3&1
\end{smallmatrix}
\right)$. When $\exp(t \lambda_k)$ acts on $ \mathcal{Z}$, the result lies not necessarily again in the big cell. But, due to the fact that the flag manifold is identified with the (complexified) homogeneous space $\mathsf{SL}(3, \C)/B$ where $B$ is the Borel subgroup of upper triangular matrices, we can always multiply (from the right) with elements from $B$ to get again an element in the big cell.
Using formula \eqref{exponential}, we compute
\begin{align*}
\exp(t\lambda_1)&= \begin{pmatrix} \cos \left(\frac{t}{2}\right) & i \sin \left(\frac{t}{2}\right) & 0 \\
i \sin \left(\frac{t}{2}\right) & \cos \left(\frac{t}{2}\right) & 0 \\
0 & 0 & 1 \\\end{pmatrix},
&& \exp(t\lambda_2) = \begin{pmatrix}
\cos \left(\frac{t}{2}\right) & \sin \left(\frac{t}{2}\right) & 0 \\
-\sin \left(\frac{t}{2}\right) & \cos \left(\frac{t}{2}\right) & 0 \\
0 & 0 & 1 \\
\end{pmatrix},
\\ 	
\exp(t\lambda_3)&= 
\begin{pmatrix} e^{\frac{i t}{2}} & 0 & 0 \\
0 & e^{-\frac{it}{2}} & 0 \\
0 & 0 & 1 \\
\end{pmatrix},
&& \exp(t\lambda_4) = \begin{pmatrix}
\cos \left(\frac{t}{2}\right) & 0 & i \sin \left(\frac{t}{2}\right) \\
0 & 1 & 0 \\
i \sin \left(\frac{t}{2}\right) & 0 & \cos \left(\frac{t}{2}\right) \\
\end{pmatrix},
\\ 
\exp(t\lambda_5)&= \begin{pmatrix}
\cos \left(\frac{t}{2}\right) & 0 & \sin \left(\frac{t}{2}\right) \\
0 & 1 & 0 \\
-\sin \left(\frac{t}{2}\right) & 0 & \cos \left(\frac{t}{2}\right) \\
\end{pmatrix},
&& \exp(t\lambda_6)=
\begin{pmatrix} 1 & 0 & 0 \\
0 & \cos \left(\frac{t}{2}\right) & i \sin \left(\frac{t}{2}\right) \\
0 & i \sin \left(\frac{t}{2}\right) & \cos \left(\frac{t}{2}\right) \\
\end{pmatrix},
\\ 
\exp(t\lambda_7)&= \begin{pmatrix} 1 & 0 & 0 \\
0 & \cos \left(\frac{t}{2}\right) &  \sin \left(\frac{t}{2}\right) \\
0 & - \sin \left(\frac{t}{2}\right) & \cos \left(\frac{t}{2}\right) \\
\end{pmatrix},
&& \exp(t\lambda_8) = 
\begin{pmatrix}
e^{\frac{i t}{2 \sqrt{3}}} & 0 & 0 \\
0 & e^{\frac{i t}{2 \sqrt{3}}} & 0 \\
0 & 0 & e^{-\frac{i t}{\sqrt{3}}} \\
\end{pmatrix}.
\end{align*}
Now we need to compute $ \exp(t \lambda_k) \cdot \mathcal{Z}$. Note that, as mentioned above, the result may not lie in the big cell. Thus we need to multiply in addition from the right with an element $\mathbf{b} = \left(\begin{smallmatrix}
				b_1	& b_2	& b_3
				\\0	& b_4& b_5
				\\0&	0	& b_6
			\end{smallmatrix} \right) \in B$, i.e., 
\begin{align*}
	\exp(\lambda_kt)\begin{pmatrix}
			1	&	0	&	0
		\\	z_1	&	1	&	0
		\\	z_2	&	z_3	&	1
		\end{pmatrix} 
			\begin{pmatrix}
				b_1	& b_2	& b_3
				\\0	& b_4& b_5
				\\0&	0	& b_6
			\end{pmatrix}
\end{align*}
in order to obtain as element of the big cell
\begin{align*}
A_k:=\begin{pmatrix}
				1	&	0	&	0\\
				f_{1, k}({z}_1,t)&1&0\\
				f_{2, k}({z}_2,t)&f_{3, k}({z}_3,t)&1
			\end{pmatrix}
\end{align*}
for some functions $f_{j, k}(z_j,t)$ with $j \in \{1, 2, 3\}$ and $k \in \{1, \dots, 8\}$ depending on the complex variables $z_j$ and the real variable $t$.
Now, for $k \in \{1, \dots, 8\}$, we solve $\exp(\lambda_kt) \mathcal{Z} \mathbf{b} = A_k$ for $\mathbf{b} \in B$. The solutions are denoted by $\beta_k\in B$ and are given as follows:
\begin{align*}
\beta_1&=\left(
\begin{array}{ccc}
\frac{1}{\cos \left(\frac{t}{2}\right)+i z_1 \sin \left(\frac{t}{2}\right)} & -i \sin \left(\frac{t}{2}\right) & 0 \\
0 & \cos \left(\frac{t}{2}\right)+i z_1 \sin \left(\frac{t}{2}\right) & 0 \\
0 & 0 & 1 \\
\end{array}
\right),
\\ \beta_2&=\left(
\begin{array}{ccc}
\frac{1}{z_1 \sin \left(\frac{t}{2}\right)+\cos \left(\frac{t}{2}\right)} & -\sin \left(\frac{t}{2}\right) & 0 \\
0 & z_1 \sin \left(\frac{t}{2}\right)+\cos \left(\frac{t}{2}\right) & 0 \\
0 & 0 & 1 \\
\end{array}
\right),
\\ \beta_3&= \left(
\begin{array}{ccc}
 e^{\frac{i t}{2}} & 0 & 0 \\
 0 & e^{-\frac{it}{2}} & 0 \\
 0 & 0 & 1 \\
\end{array}
\right),\\ \beta_4&=\left(
\begin{array}{ccc}
\frac{1}{\cos \left(\frac{t}{2}\right)+i z_2 \sin \left(\frac{t}{2}\right)} & \frac{z_3 \sin \left(\frac{t}{2}\right)}{-z_2 \sin \left(\frac{t}{2}\right)+z_1 z_3 \sin \left(\frac{t}{2}\right)+i \cos \left(\frac{t}{2}\right)} & -i \sin \left(\frac{t}{2}\right) \\
0 & 1-\frac{z_1 \left(z_3 \sin \left(\frac{t}{2}\right)\right)}{-z_2 \sin \left(\frac{t}{2}\right)+z_1 z_3 \sin \left(\frac{t}{2}\right)+i \cos \left(\frac{t}{2}\right)} & iz_1 \sin \left(\frac{t}{2}\right) \\
0 & 0 & \cos \left(\frac{t}{2}\right)+i \left(z_2-z_1 z_3\right) \sin \left(\frac{t}{2}\right) \\
\end{array}
\right),
\\\beta_5&=\left(
\begin{array}{ccc}
\frac{1}{z_2 \sin \left(\frac{t}{2}\right)+\cos \left(\frac{t}{2}\right)} & -\frac{z_3 \sin \left(\frac{t}{2}\right)}{z_2 \sin \left(\frac{t}{2}\right)-z_1 z_3 \sin \left(\frac{t}{2}\right)+\cos \left(\frac{t}{2}\right)} & -\sin \left(\frac{t}{2}\right) \\
0 & \frac{z_2 \sin \left(\frac{t}{2}\right)+\cos \left(\frac{t}{2}\right)}{z_2 \sin \left(\frac{t}{2}\right)-z_1 z_3 \sin \left(\frac{t}{2}\right)+\cos \left(\frac{t}{2}\right)} & z_1 \sin \left(\frac{t}{2}\right) \\
0 & 0 & z_2 \sin \left(\frac{t}{2}\right)-z_1 z_3 \sin \left(\frac{t}{2}\right)+\cos \left(\frac{t}{2}\right) \\
\end{array}
\right),
\\ \beta_6&=\left(
\begin{array}{ccc}
1 & 0 & 0 \\
0 & \frac{1}{\cos \left(\frac{t}{2}\right)+i z_3 \sin \left(\frac{t}{2}\right)} & -\frac{i}{\sin \left(\frac{t}{2}\right)+\cos \left(\frac{t}{2}\right) \cot \left(\frac{t}{2}\right)} \\
0 & 0 & \cos \left(\frac{t}{2}\right)+i z_3 \sin \left(\frac{t}{2}\right) \\
\end{array}
\right),
\\
\beta_7&=\left(
\begin{array}{ccc}
1 & 0 & 0 \\
0 & \frac{1}{z_3 \sin \left(\frac{t}{2}\right)+\cos \left(\frac{t}{2}\right)} & -\frac{1}{\sin \left(\frac{t}{2}\right)+\cos \left(\frac{t}{2}\right) \cot \left(\frac{t}{2}\right)} \\
0 & 0 & z_3 \sin \left(\frac{t}{2}\right)+\cos \left(\frac{t}{2}\right) \\
\end{array}
\right),
\\ \beta_8&=\left(
\begin{array}{ccc}
e^{-\frac{i t}{2 \sqrt{3}}} & 0 & 0 \\
0 & e^{-\frac{i t}{2 \sqrt{3}}} & 0 \\
0 & 0 & e^{\frac{i t}{\sqrt{3}}} \\
\end{array}
\right).
\end{align*}
In order to obtain the vector fields, we must compute the derivatives of $A_k$ with respect to $t$ and evaluate in $t=0$. This reduces to determining the derivatives of the coordinate functions 
$$\frac{d}{dt}f_{j, k}({z}_j,t)$$ in $t=0$ for all $j\in \{1,2,3\}$ and $k \in \{1, \dots, 8\}$. The corresponding vector fields are then, for $k \in \{1, \dots, 8\}$, given by 
\begin{align*}
X_{\lambda_k}=
		\left(\frac{d}{dt}\bigg|_{t=0}f_{1, k}({z}_1,t)\right)\partial_{z_1} +	
		\left(\frac{d}{dt}\bigg|_{t=0}f_{2, k}({z}_2,t)\right)\partial_{z_2}+	
		\left(\frac{d}{dt}\bigg|_{t=0}f_{3, k}({z}_3,t)\right)\partial_{z_3}
\end{align*}
which yields the claim.
\end{proof}

\noindent 
Now we compute the explicit formula for the momentum map on $  \mathcal{O}^{\mathsf{SU}(3)} \simeq \mathbb{F}_{1,2} (\C^3) \simeq \mathsf{SL}(3,\C)/B$ associated with vortex dynamics.

\begin{theorem}
\label{momentmap}
Let $K_1  = 1 + |z_1|^2 + |z_2|^2$ and $K_2  = 1 + |z_3|^2 + |z_1z_3-z_2|^2$.
The momentum map for the left action of $\mathsf{SU}(3)$ on the generic coadjoint orbit $\mathsf{SL}(3,\C)/B$ is given by
	\begin{align*}
		\mu: \mathsf{SL}(3,\C)/ B \to \mathfrak{su}(3)^*,  \qquad
		\begin{pmatrix}
			1	&	0	&	0
		\\	z_1	&	1	&	0
		\\	z_2	&	z_3	&	1
		\end{pmatrix}  
		\mapsto {i}(\mu_{ij})_{1 \leq i, j \leq 3}
	\end{align*}
	where $(\mu_{ij})_{1 \leq i, j \leq 3}$ is the traceless, anti-Hermitian matrix with entries
	\begin{align*}
	\mu_{11}&=\frac{1}{3} \left(\frac{x_3^2+y_3^2+2}{K_2}-\frac{x_2^2+y_2^2-1}{K_1}\right),
	\\
	\mu_{22}&=\frac{1}{3} \left(-\frac{2 x_2^2+2 y_2^2+1}{K_1}-\frac{x_3^2+y_3^2-1}{K_2}\right),
	\\ 
	\mu_{33}&=-(\mu_{11}+\mu_{22}),
	\\
	\mu_{12}&=\frac{\left(i y_1-x_1\right) \left(x_3-i y_3\right)-i y_2+x_2}{K_2}-\frac{x_1-i y_1}{K_1},
	\\
	\mu_{13}&=\frac{\left(i y_1-x_1\right) \left(x_3-i y_3\right)-i y_2+x_2}{K_2}-\frac{x_1-i y_1}{K_1},
	\\
	\mu_{23}&=\frac{i y_3+x_3}{K_2}-\frac{\left(x_1+i y_1\right) \left(x_2-i y_2\right)}{K_1}.
\end{align*}
The remaining entries are determined by the fact that the matrix is anti-Hermitian.
\end{theorem}

\noindent 

\begin{proof}
By definition of the moment map, we must have 
\begin{align*}
	d\langle \mu,\lambda_k \rangle  = \iota_{X_{\lambda_k}}\omega,
\end{align*}
for all $\lambda_k\in\mathfrak{g}$ and induced vector fields $X_{\lambda_k}$ from Lemma \ref{lem:infinitesimalVF}. 
Moreover, recall from \eqref{ex:explicitTrace} that the dual pairing is given by the trace. Therefore we have
    $$\langle \mu,\lambda_k\rangle: \mathsf{SL}(3,\C)/B \to \R,\quad x\mapsto \langle \mu(x),\lambda_k \rangle = \trace(\mu(x)\lambda_k). $$ 

\noindent 
The dual pairing explicitly becomes
\begin{align*}
	&\trace\left(\left(
	\begin{array}{ccc}
	\mu_{11} & \mu_{12} & \mu_{13} \\
	\mu_{21} & \mu_{22} & \mu_{23} \\
	\mu_{31} & \mu_{32} & \mu_{33} \\
	\end{array}
	\right)\left(
\begin{array}{ccc}
 0 & \frac{i}{2} & 0 \\
 \frac{i}{2} & 0 & 0 \\
 0 & 0 & 0 \\
\end{array}
\right)\right)&&=\frac{i \mu_{12}}{2}+\frac{i \mu_{21}}{2},
\\ 
&\trace\left(\left(
\begin{array}{ccc}
\mu_{11} & \mu_{12} & \mu_{13} \\
\mu_{21} & \mu_{22} & \mu_{23} \\
\mu_{31} & \mu_{32} & \mu_{33} \\
\end{array}
\right)\left(
\begin{array}{ccc}
 0 & \frac{1}{2} & 0 \\
 -\frac{1}{2} & 0 & 0 \\
 0 & 0 & 0 \\
\end{array}
\right)\right)&&=\frac{\mu_{21}}{2}-\frac{\mu_{12}}{2},
\\ 
&\trace\left(\left(
\begin{array}{ccc}
\mu_{11} & \mu_{12} & \mu_{13} \\
\mu_{21} & \mu_{22} & \mu_{23} \\
\mu_{31} & \mu_{32} & \mu_{33} \\
\end{array}
\right)\left(
\begin{array}{ccc}
 \frac{i}{2} & 0 & 0 \\
 0 & -\frac{i}{2} & 0 \\
 0 & 0 & 0 \\
\end{array}
\right)\right)&&=\frac{i \mu_{11}}{2}-\frac{i \mu_{22}}{2},
\\ 
&\trace\left(\left(
\begin{array}{ccc}
\mu_{11} & \mu_{12} & \mu_{13} \\
\mu_{21} & \mu_{22} & \mu_{23} \\
\mu_{31} & \mu_{32} & \mu_{33} \\
\end{array}
\right)\left(
\begin{array}{ccc}
 0 & 0 & \frac{i}{2} \\
 0 & 0 & 0 \\
 \frac{i}{2} & 0 & 0 \\
\end{array}
\right)\right)&&=\frac{i \mu_{13}}{2}+\frac{i \mu_{31}}{2},
\\ 
&\trace\left(\left(
\begin{array}{ccc}
\mu_{11} & \mu_{12} & \mu_{13} \\
\mu_{21} & \mu_{22} & \mu_{23} \\
\mu_{31} & \mu_{32} & \mu_{33} \\
\end{array}
\right)\left(
\begin{array}{ccc}
 0 & 0 & \frac{1}{2} \\
 0 & 0 & 0 \\
 -\frac{1}{2} & 0 & 0 \\
\end{array}
\right)\right)&&=\frac{\mu_{31}}{2}-\frac{\mu_{13}}{2},
\\ 
&\trace\left(\left(
\begin{array}{ccc}
\mu_{11} & \mu_{12} & \mu_{13} \\
\mu_{21} & \mu_{22} & \mu_{23} \\
\mu_{31} & \mu_{32} & \mu_{33} \\
\end{array}
\right)\left(
\begin{array}{ccc}
 0 & 0 & 0 \\
 0 & 0 & \frac{i}{2} \\
 0 & \frac{i}{2} & 0 \\
\end{array}
\right)\right)&&=\frac{i \mu_{23}}{2}+\frac{i \mu_{32}}{2},
\\ 
&\trace\left(\left(
\begin{array}{ccc}
\mu_{11} & \mu_{12} & \mu_{13} \\
\mu_{21} & \mu_{22} & \mu_{23} \\
\mu_{31} & \mu_{32} & \mu_{33} \\
\end{array}
\right)\left(
\begin{array}{ccc}
 0 & 0 & 0 \\
 0 & 0 & \frac{1}{2} \\
 0 & -\frac{1}{2} & 0 \\
\end{array}
\right)\right)&&=\frac{\mu_{32}}{2}-\frac{\mu_{23}}{2},
\\ 
&\trace\left(\left(
\begin{array}{ccc}
\mu_{11} & \mu_{12} & \mu_{13} \\
\mu_{21} & \mu_{22} & \mu_{23} \\
\mu_{31} & \mu_{32} & \mu_{33} \\
\end{array}
\right)\left(
\begin{array}{ccc}
 \frac{i}{2 \sqrt{3}} & 0 & 0 \\
 0 & \frac{i}{2 \sqrt{3}} & 0 \\
 0 & 0 & -\frac{i}{\sqrt{3}} \\
\end{array}
\right)\right)&& =\frac{i \left(\mu_{11}+\mu_{22}-2 \mu_{33}\right)}{2 \sqrt{3}}.
\end{align*}
In Lemma \ref{hermitianmetric}, we obtained the Hermitian metric $h=(h_{kl})_{1 \leq k,l \leq n}$ on the flag manifold. Moreover, the (real) symplectic form is given by $\omega=\frac{i}{2}\sum_{k,l=1}^nh_{kl}dz_k\wedge d\overline{z}_l$. In terms of real coordinates $(x_1,x_2,x_3,y_1,y_2,y_3)\in\R^6$ the matrix representing the symplectic form is given by 
\begin{align*}
	\omega:=\begin{pmatrix}
		\Im(h)	& -\Re(h)
		\\ 
		\Re(h) & \Im(h)
	\end{pmatrix}
\end{align*}
with 
\begin{align*}
	\Im(h) =\begin{pmatrix}
 0 & \frac{x_2 y_1-x_1 y_2}{K_1^2}-\frac{y_3 \left(x_3^2+y_3^2+1\right)}{K_2^2} & \frac{-\left(x_3 \left(y_1-2 x_2 y_3\right)\right)-x_3^2 y_2+y_3 \left(x_1+y_2 y_3\right)}{K_2^2} \\
 \frac{x_1 y_2-x_2 y_1}{K_1^2}+\frac{y_3 \left(x_3^2+y_3^2+1\right)}{K_2^2} & 0 & \frac{x_3 y_2-x_2 y_3+y_1}{K_2^2} \\
 \frac{x_3^2 y_2+x_3 \left(y_1-2 x_2 y_3\right)-y_3 \left(x_1+y_2 y_3\right)}{K_2^2} & -\frac{x_3 y_2-x_2 y_3+y_1}{K_2^2} & 0 \\
\end{pmatrix}
\end{align*}
and \begin{align*}
	\Re(h)= \begin{pmatrix}
 \frac{x_2^2+y_2^2+1}{K_1^2}+\frac{x_3^2 \left(2 y_3^2+1\right)+x_3^4+y_3^4+y_3^2}{K_2^2} & -\frac{x_1 x_2+y_1 y_2}{K_1^2}-\frac{x_3 \left(x_3^2+y_3^2+1\right)}{K_2^2} & \frac{y_3 \left(2 x_3 y_2+y_1\right)+x_2 \left(x_3^2-y_3^2\right)+x_1 x_3}{K_2^2} \\
 -\frac{x_1 x_2+y_1 y_2}{K_1^2}-\frac{x_3 \left(x_3^2+y_3^2+1\right)}{K_2^2} & \frac{x_1^2+y_1^2+1}{K_1^2}+\frac{x_3^2+y_3^2+1}{K_2^2} & -\frac{x_1+x_2 x_3+y_2 y_3}{K_2^2} \\
 \frac{y_3 \left(2 x_3 y_2+y_1\right)+x_2 \left(x_3^2-y_3^2\right)+x_1 x_3}{K_2^2} & -\frac{x_1+x_2 x_3+y_2 y_3}{K_2^2} & \frac{K_1}{K_2^2} \\
	\end{pmatrix}.
\end{align*}
Evaluating $d\langle \mu,\lambda_k \rangle  = \iota_{X_{\lambda_k}}\omega$ using the matrix representing $\omega$, we obtain the equations
\begin{align}
\label{systemmoment}
	\begin{cases}
		d\left(\frac{i \mu_{12}}{2}+\frac{i \mu_{21}}{2}\right)&=\iota_{X_{\lambda_1}}\omega,
\\ 	
d\left(\frac{\mu_{21}}{2}-\frac{\mu_{12}}{2}\right)&=\iota_{X_{\lambda_2}}\omega,
\\ 		
d\left(\frac{i \mu_{13}}{2}+\frac{i \mu_{31}}{2}
\right)&=\iota_{X_{\lambda_4}}\omega,
\\ 	
d\left(\frac{\mu_{31}}{2}-\frac{\mu_{13}}{2}\right)&=\iota_{X_{\lambda_5}}\omega,
\\		
d\left(\frac{i \mu_{23}}{2}+\frac{i \mu_{32}}{2} \right)&=\iota_{X_{\lambda_6}}\omega,
\\ 	
d\left(\frac{\mu_{32}}{2}-\frac{\mu_{23}}{2}\right)&=\iota_{X_{\lambda_7}}\omega,
\\		
d\left(\frac{i \mu_{11}}{2}-\frac{i \mu_{22}}{2}\right)&=\iota_{X_{\lambda_3}}\omega,
\\ 	
d\left(\frac{i \left(\mu_{11}+\mu_{22}-2 \mu_{33}\right)}{2 \sqrt{3}}
\right)&=\iota_{X_{\lambda_8}}\omega
	\end{cases}
\end{align}
where we need to solve for the components $\mu_{ij}$ of the momentum map. The 1-forms on the left hand side are exact and therefore of the general form $dF = \sum_{k=1}^3 \frac{\partial F}{\partial x_k}dx_k + \sum_{k=1}^3 \frac{\partial F}{\partial y_k}dy_k $. The 1-forms on the right hand side are contractions and of the general form $\sum_{k=1}^3 G_k dx_k + \sum_{k=1}^3 G_{3+k}dy_k$.
Thus we must solve $\frac{\partial F}{\partial x_k} =  G_k$ and $ \frac{\partial F}{\partial y_k} = G_{k+3}$ for all $k \in \{1, 2, 3\}$. 
We now proceed as in Example \ref{ex:exactOneForm}. Considering the coordinate $x_1$, we find
\begin{align*}
	F(x_1,x_2,x_3,y_1,y_2,y_3)=\int G_1(x_1,  x_2,x_3,y_1,y_2,y_3) dx_1 + C(x_2,x_3,y_1,y_2,y_3)
\end{align*}
and obtain therefore an expression for $F$. Using $\frac{\partial F}{\partial x_k} = G_2$ we obtain
\begin{align*}
	G_2 (x_1,x_2,x_3,y_1,y_2,y_3) & = \frac{\partial F}{\partial x_2}(x_1,x_2,x_3,y_1,y_2,y_3) \\
	& =\frac{\partial}{\partial x_2} \int G_1(x_1,  x_2,x_3,y_1,y_2,y_3) dx_1 +\frac{\partial}{\partial x_2} C(x_2,x_3,y_1,y_2,y_3)
\end{align*}
and therefore 
$$
\frac{\partial}{\partial x_2} C = G_2 - \frac{\partial}{\partial x_2} \int G_1
$$
which yields 
$$
C =  \int \left( G_2 - \frac{\partial}{\partial x_2} \int G_1 \right) dx_2 + C'(x_3, y_1, y_2, y_3).
$$
Iterating this procedure, we determine $F$ in terms of $G_1, \dots, G_6$. Therefore, if $G_1, \dots, G_6$ are explicitly given, we can find an explicit formula for $F$. 

Now we will apply this procedure to the systems of equations given in \eqref{systemmoment}. 
We start with the coupled system
\begin{align*}
\begin{cases}   
 d\left(\frac{i \mu_{12}}{2}+\frac{i \mu_{21}}{2}
\right) & =\iota_{X_{\lambda_1}}\Omega,
\\ 	
 d\left(\frac{\mu_{21}}{2}-\frac{\mu_{12}}{2}\right) & =\iota_{X_{\lambda_2}}\Omega.
\end{cases}
\end{align*}
By integration we obtain the following expressions. Note that, in our situation, the function $C$ from above is constant and may be chosen to be zero.
\begin{align*}
\frac{i}{2}\left(\mu_{12}+\mu_{21}\right)&=\frac{1}{2} \left(-\frac{x_1 \left(x_3^2+y_3^2\right)-x_2 x_3-y_2 y_3}{K_2}-\frac{x_1}{K_1}\right)=: \alpha,
\\ \frac{1}{2}\left(\mu_{21}-\mu_{12}\right)&=\frac{1}{2} \left(\frac{x_3 y_2-x_3^2 y_1-y_3 \left(x_2+y_1 y_3\right)}{K_2}-\frac{y_1}{K_1}\right)=: \beta.
\end{align*}
We now can solve for $\mu_{21} $ via 
\begin{align*}
i\mu_{21}=\alpha+i\beta&=\frac{1}{2} \left(\frac{-\left(x_1+i y_1\right) \left(x_3+i y_3\right)+x_2+i y_2}{K_2}-\frac{x_1+i y_1}{K_1}\right)
\\&=\frac{1}{2} \left(-\frac{z_1}{K_1}+\frac{-z_1z_3+z_2}{K_2}\right)
\end{align*}
and find 
\begin{align}
	{\mu_{21}=-\frac{i}{2} \left(-\frac{z_1}{K_1}+\frac{-z_1z_3+z_2}{K_2}\right)}.
\end{align} 
The second pair of coupled equations is
\begin{align*}
\begin{cases}
		d\left(\frac{i \mu_{13}}{2}+\frac{i \mu_{31}}{2}
\right) & = \iota_{X_{\lambda_4}}\Omega,
\\ 	d\left(\frac{\mu_{31}}{2}-\frac{\mu_{13}}{2}\right) & =\iota_{X_{\lambda_5}}\Omega
\end{cases}
\end{align*}
which can be integrated as\begin{align*}
	\frac{i \mu_{13}}{2}+\frac{i \mu_{31}}{2}&=\frac{1}{2} \left(\frac{-x_2+x_1 x_3-y_1 y_3}{K_2}-\frac{x_2}{K_1}\right) =: \gamma, 
	\\ \frac{\mu_{31}}{2}-\frac{\mu_{13}}{2}&=\frac{1}{2} \left(\frac{x_3 y_1+x_1 y_3-y_2}{K_2}-\frac{y_2}{K_1}\right) =: \delta
\end{align*}
We obtain
\begin{align*}
	i\mu_{31}=\gamma+i\delta=-\frac{x_2+{\color{black}i}y_2}{2 K_1}-\frac{-x_3 y_1-x_1 \left(x_3+y_3\right)+x_2+y_2+y_1 y_3}{2 K_2}
\end{align*}
and finally
\begin{align*}
	\mu_{31}=\frac{-i}{2}\left(-\frac{x_2+i y_2}{K_1}-\frac{-\left(x_1+i y_1\right) \left(x_3+i y_3\right)+x_2+i y_2}{K_2}\right).
\end{align*}
%And thus in complex coordinates:
%\begin{align}
%{	\mu_{31}=\frac{-i}{2}\left(-\frac{z_2}{K_1}-\frac{-z_1 z_3+z_2}{K_2}\right)}
%\end{align}
The next pair of equations is
\begin{align*}
\begin{cases}
		d\left(\frac{i \mu_{23}}{2}+\frac{i \mu_{32}}{2}
\right) & =\iota_{X_{\lambda_6}}\Omega, 
\\ 	d\left(\frac{\mu_{32}}{2}-\frac{\mu_{23}}{2}\right) & =\iota_{X_{\lambda_7}}\Omega.
\end{cases}
\end{align*}
We compute
\begin{align*}
	\frac{i \mu_{23}}{2}+\frac{i \mu_{32}}{2}&=\frac{1}{2} \left(\frac{x_3}{K_2}-\frac{x_1 x_2+y_1 y_2}{K_1}\right)=: \zeta, 
\\ \frac{\mu_{32}}{2}-\frac{\mu_{23}}{2}&=\frac{1}{2} \left(\frac{x_2 y_1-x_1 y_2}{K_1}-\frac{y_3}{K_2}\right)=: \eta
\end{align*}
and obtain $\zeta+i \eta=i\mu_{32}$ and therefore
\begin{align*}
	i\mu_{32}&=\frac{1}{2} \left(\frac{x_3}{K_2}-\frac{x_1 x_2+y_1 y_2}{K_1}\right)+i\frac{1}{2} \left(\frac{x_2 y_1-x_1 y_2}{K_1}-\frac{y_3}{K_2}\right)
	\\&=\frac{1}{2}\left(\frac{-x_1x_2-y_1y_2+i(x_2y_1-x_1y_2)}{K_1}+\frac{x_3-iy_3}{K_2}
	\right).
\end{align*}
%Using the fact the first numerator can be written as $-\left(x_1-i y_1\right) \left(x_2+i y_2\right)$ which also equals $-\overline{z}_1z_2$ we obtain the closed form of $\mu_{32}$:
%\begin{align}
%{	\mu_{32}=-\frac{i}{2}\left(-\frac{\overline{z}_1z_2}{K_1}+\frac{\overline{z}_3}{K_2}\right)
%}\end{align}
In order to determine $\mu_{33}$, we integrate the last equation in \eqref{systemmoment} and obtain
\begin{align*}
	\frac{i \left(\mu_{11}+\mu_{22}-2 \mu_{33}\right)}{2 \sqrt{3}}=\frac{\sqrt{3} }{4} \left(\frac{x_2^2+y_2^2}{K_1}-\frac{1}{K_2}\right).
\end{align*}
Using the fact that the matrix $M$ is traceless, i.e. $\mu_{11}+\mu_{22}+\mu_{33}=0$, the left-hand side reduces to
\begin{align*}
	\frac{-3i\mu_{33}}{2\sqrt{3}}=\frac{\sqrt{3} }{4} \left(\frac{x_2^2+y_2^2}{K_1}-\frac{1}{K_2}\right)
\end{align*}
and therefore
\begin{align*}
	\mu_{33}=\frac{i}{2}\left(\frac{x_2^2+y_2^2}{K_1}-\frac{1}{K_2}\right).
\end{align*}
On the other hand, using $2(\mu_{11}+\mu_{22})=-2\mu_{33}$, we have
\begin{align*}
\frac{i}{2\sqrt{3}}	\left(\mu_{11}+\mu_{22}-2 \mu_{33}\right)=\frac{3i}{2\sqrt{3}}	\left(\mu_{11}+\mu_{22}\right)=\frac{\sqrt{3} }{4} \left(\frac{x_2^2+y_2^2}{K_1}-\frac{1}{K_2}\right).
\end{align*}
Moreover, integrating the remaining Cartan equation $d\left(\frac{i \mu_{11}}{2}-\frac{i \mu_{22}}{2}\right)=\iota_{X_{\lambda_3}}\omega$
gives \[ \frac{i}{2}(\mu_{11}-\mu_{22})=\kappa. \]
This yields the system of equations
\begin{align*}
\begin{cases}
\frac{3i}{2\sqrt{3}}	\left(\mu_{11}+\mu_{22}\right)&=\lambda, 
\\
	\frac{i}{2}\left({\mu_{11}}-{\mu_{22}}\right)&=\kappa 
\end{cases}
\quad \Longleftrightarrow\quad
\begin{cases}
\frac{3i}{2\sqrt{3}}	\left(\mu_{11}+\mu_{22}\right)&=\lambda, 
\\
	\frac{3i}{2\sqrt{3}}\left({\mu_{11}}-{\mu_{22}}\right)&= \frac{3}{\sqrt{3}}\kappa.
\end{cases}
\end{align*}	
Adding (resp. subtracting) the equations gives 
\begin{align*}
	\mu_{11}&=-\frac{\sqrt{3}}{3}i\left(\alpha+\frac{3}{\sqrt{3}}\beta\right)=-\frac{i}{\sqrt{3}}\left(\frac{x_2^2+y_2^2-1}{2 \sqrt{3} K_1}-\frac{x_3^2+y_3^2+2}{2 \sqrt{3} K_2}\right) ,
	\\
\mu_{22}&=-\frac{\sqrt{3}}{3}i\left(\alpha-\frac{3}{\sqrt{3}}\beta\right)=-\frac{i}{\sqrt{3}}\left(\frac{2 x_2^2+2 y_2^2+1}{2 \sqrt{3} K_1}+\frac{x_3^2+y_3^2-1}{2 \sqrt{3} K_2}\right),
\\ 
\mu_{33}&=-(\mu_{11}+\mu_{22}).
\end{align*}
Now as we are working over $\mathfrak{su}^*(3)$ we have obtained all the entries of the matrix.
\end{proof}

%%%%%%%%%%%%%%%%%%%%%%%%%%%%%%%%%%%%%%%
%%%%%%%%%%%% new subsection %%%%%%%%%%%%%

%\subsection{The momentum map on $\mathbb{CP}^3$}
%\label{momentumMapCPthree}
%
%\mbox{ \ }
%\todo{\ssc{where does the following formula come from?}}
%\ggc{The weighted momentum mapping of the $\mathsf{SU}(4)$-action on $\mathbb{CP}^3$ is given by
%\begin{align}
%J:(\mathbb{CP}^3)^n \to \mathfrak{su}(4)^*, \quad (z_1,z_2,\dots,z_n) \mapsto \sum_{k=1}^n \Gamma_k \frac{z_k\otimes \overline{z}_k}{|z_k|^2}-\frac{1}{4}I\sum_{k=1}^n \Gamma_k,
%\end{align}
%where $\Gamma_1, \dots, \Gamma_n \in \R^{\neq 0}$ are the vortex strengths.}
%\todo{\ssc{what is $z_k\otimes \overline{z}_k$ ? what is $I$ in from of the second sum? Does the 2nd sum on purpose does not contain any $z_k$ ?}}
%\todo{\ssc{what is the point/aim of the following sentence?}}
%\ggc{For $n=3$, the action of $\mathsf{SU(4)}$ on $\mathbb{CP}^3\times \mathbb{CP}^3\times\mathbb{CP}^3$ is connected to the reduced space $M=(\mathbb{CP}^3\times \mathbb{CP}^3\times\mathbb{CP}^3)/\mathsf{SU}(4)$.}

%%%%%%%%%%%%%%%%%%%%%%%%%%%%%%%%%%%%%%%%%%%%%%%%%%%%%%%%
%%%%%%%%%%%  new section %%%%%%%%%%%%%%%%%%%%%%%%%%%%%%%
%%%%%%%%%%%%%%%%%%%%%%%%%%%%%%%%%%%%%%%%%%%%%%%%%%%%%%%%%%%

\section{Green's function and the vortex Hamiltonian for $\mathbb{CP}^n$}
\label{sectionGreenAndHam}
In Section \ref{sectionMomentumMap}, we recalled the point vortex momentum map on $\mathbb{CP}^2$ and computed the one on $\mathbb{F}_{1,2}(\C^3)$, i.e., on the degenerate and generic orbits of the action of $\mathsf{SU}(3)$.
The aim of this section is to determine the Hamiltonian for the point vortex problem on $\mathbb{CP}^2$. An important ingredient here is the fundamental solution of the Laplace-Beltrami operator.

%%%%%%%%%%%%%%%%%%%%%%%%%%%%%%%%%%%%%%%%%
%%%%%%%%%%%  new subsection  %%%%%%%%%%%%

\subsection{The Laplace-Beltrami operator}
Let $(M,g)$ be a Riemannian manifold of dimension $n$ with local coordinates $x=(x_1,\dots,x_n)$ in which we express the metric by the symmetric $(n\times n)$-matrix $(g_{ij}):= (g_{ij})_{1 \leq i, j \leq n}$. 
The inverse of this matrix is denoted by $(g_{ij})^{-1} =:(g^{ij})$. 
The Riemannian volume form on $(M, g)$ is given in local coordinates by 
\[ \sqrt{|\det(g)|}dx_1\wedge \cdots \wedge dx_n =: d\mu. \]
Let $f: M \to \R$ a smooth function. Then the \textit{Laplace-Beltrami operator} $\Delta$ on $M$ is given in local coordinates $(x_1,\dots,x_n)$ by
\begin{align*}
%\label{laplace}
	\Delta f= -\frac{1}{\sqrt{|\det(g_{ij})|}}\sum_{i,j =1}^n\partial_{x_j}g^{ij}\sqrt{|\det(g_{ij})|} \ \partial_{x_i}f.
\end{align*}
which, on $\R^n$ equipped with the Euclidean metric, yields $\Delta=\sum_{j=1}^n\partial_{x_j}^2$.
The {\em big diagonal} $\Diag_N(M)$ of the $N$-fold product $ M \times \dots \times M$ is defined as 
$$\Diag_N(M):= \{(y_1, \dots, y_N) \in M^N \mid  y_j = y_k \mbox{ for some } j \neq k \mbox{ with } 1 \leq j,k \leq N \},$$
i.e., it consists of all $N$-tuple point on $M$ for which at least two points coincide.
If $(M, g)$ is compact, then, for the Laplace-Beltrami operator, there exists a function $G: (M \times M) \setminus {\Diag_2(M)} \to \R$, referred to as \textit{Green's function},
satisfying the following properties (see for instance Section 2.3 in Aubin \cite{aubin1998some}):
\begin{enumerate}
	\item For all functions $\phi\in C^2(M, \R)$, we have
	\begin{align}
	\label{Gr}
	\phi(p) = \frac{1}{\vol(M)}\int_M \phi(q) \ d\mu(q) +\int_MG(p,q)\Delta\phi(q) \ d\mu(q).
	\end{align}
	\item $G$ is smooth on $(M\times M) \setminus \Diag_2(M)$.
	\item $G$ is symmetric, i.e. $G(p,q)=G(q,p)$ for all $p, q \in M$.
	\item We have $\int_M G(p,q) \ d\mu(q) = \text{constant}$ for all $p \in M$.
\end{enumerate}
Green's function is also called a {\em fundamental solution} of the Laplace-Beltrami operator.

%%%%%%%%%%%% begin purge
\purge{
\begin{notation}
	The Dirac $\delta$-function (which is in fact a distribution) is given by 
	\begin{align*}
		\delta(x) = \begin{cases}
			0		&\quad x\neq 0\\
			\infty	&\quad x = 0
		\end{cases}
	\end{align*}
	with the following properties:
	\begin{align*}
		\int_{-\infty}^{\infty} \delta(x)dx = 1 \quad\text{and}\quad \int_{-\infty}^{\infty}f(x)\delta(x-a)dx =f(a),
	\end{align*}
	for a continuous function $f$ in $x=a$.
\end{notation}

\begin{definition}
Let $(M,g)$ be a compact Riemannian manifold with diagonal $D_M=\{(x,x)\in M\times M \mid x\in M\}$. The \textit{Green's function} $G(p,q)$ is the function $$G:(M\times M)\setminus D_M\to \mathbb{R}$$ satisfying the condition:
\begin{enumerate}[(i)]
	\item $\Delta_qG(p,q) = \delta_p(q) = \delta(p-q)$ if $M$ has boundary,
	\item $\Delta_qG(p,q) = \delta_p(q) - \operatorname{vol}(M)^{-1}$
\end{enumerate}
where $\operatorname{vol}(M)$ is the volume of the manifold $M$ and the Laplace operator acts in a distributional sense.
\end{definition}

}
%%%%%%%%%%%% end purge
 
 \begin{examples}
\mbox{ \ }
	\begin{enumerate}
		\item 
		Green's function on $\R^2$ equipped with the Euclidean metric is given by 
		$$G(x,y)=-\frac{1}{2\pi}\ln|x-y|$$
		{where $|x-y|:= \sqrt{(x_1-y_1)^2 + (x_2 -y_2)^2}$ is the Euclidean distance. Note that} $|x-y|$ coincides with the geodesic distance $r(x,y)$ of the Euclidean metric.
		
		\item 
		Consider the unit sphere $\mathbb{S}^2$ with $x=(\sin\theta\cos\phi, \ \sin\theta\sin \phi, \ \cos\theta)\in \mathbb{S}^2$ and $y=(\sin\theta'\cos\phi', \ \sin\theta'\sin \phi', \ \cos\theta')\in \mathbb{S}^2$. Then Green's function $G(x,y)$ for the spherical Laplace operator is given by (see Dritschel \cite{dritschel1988contour}) \[ G(x,y)=\frac{1}{2\pi}\log(1-\cos\Theta), \]
		where $\Theta=\cos \theta \cos \theta'+ \sin \theta \sin \theta' \cos(\phi-\phi')$.
	\end{enumerate}
\end{examples}

\noindent
On the flag manifold $\mathbb{F}_{1,2}(\mathbb{C}^3)$, we have 

\begin{prop}
The Laplace operator on the flag manifold is given by 
\[ \Delta_{\mathbb{F}_{1,2}(\mathbb{C}^3)} = \Delta_{\mathbb{CP}^2} + \Delta_R \]
where 
\begin{align*}
	\Delta_{\mathbb{CP}^2} = \sum_{j,k=1}^2(1+\delta_{jk}{z_k}{\overline{z}_j})\partial_{z_j}\partial_{{\overline{z}}_k}
\end{align*}
is the Laplace operator on $\mathbb{CP}^2$ and $ \Delta_R$ a correction term given by
%\[ K_1\left((1+|z_1|^2)\partial_{z_1}\partial_{\overline{z}_1} + \overline{z}_1z_2 \partial_{z_1}\partial_{\overline{z}_2} + z_1\overline{z}_2\partial_{z_2}\partial_{\overline{z}_1}+ (1+|z_2|^2)\partial_{z_2}\partial_{\overline{z}_2}\right) \] 
{\small
\begin{align*}
\Delta_R & = \frac{K_1^2}{K_2}\left( \partial_{z_1}\partial_{\overline{z}_1}+z_3\partial_{z_1}\partial_{\overline{z}_2}+\overline{z}_3\partial_{z_2}\partial_{\overline{z}_1}+	|z_3|^2 \partial_{z_2}\partial_{\overline{z}_2}\right) +\left(K_1(1+|z_3|^2)+\frac{K_2^2}{K_1}\right)\partial_{z_3}\partial_{\overline{z}_3} \\ 
& \quad + (\overline{z}_1+z_3\overline{z}_2)(\overline{z}_1z_2-z_3-z_3|z_1|^2)\partial_{z_1}\partial_{\overline{z}_3} 
+ (\overline{z}_1+\overline{z}_2z_3)\left((1+|z_2|^2)-z_1\overline{z}_2z_3\right)\partial_{z_2}\partial_{\overline{z}_3}
\\
& \quad +(z_1+\overline{z}_3z_2)(z_1\overline{z}_2-\overline{z}_3-z_3|z_1|^2)\partial_{z_3}\partial_{\overline{z}_1}
 +({z}_1+{z}_2\overline{z}_3)\left((1+|z_2|^2)-{\overline{z}}_1{z}_2\overline{z}_3\right)\partial_{z_3}\partial_{\overline{z}_2}.
\end{align*}
}
where $K_1$ and $K_2$ are the functions from Lemma \ref{lem:potentialOnOrbits}.
\end{prop}

\begin{proof}
On an $n$-dimensional K\"ahler manifold $(M,h)$ with K\"ahler potential $K_M$, the Laplace operator $\Delta$ is given by 
	\begin{align*}
	\Delta  = 2 \sum_{i, j=1}^n h^{ij}\partial_{i} \overline{\partial}_j K_{{M}}. 
	\end{align*}
Using the expression for the K\"ahler potential on the flag manifold from Lemma \ref{lem:potentialOnOrbits}, one obtains the expression in the proposition by direct calculations.
\end{proof}

%%%%%%%%%%%%%%%%%%%%%%%%%%%%%%%%%%%%%%%%%%%%%%%
%%%%%%%%%% subsection %%%%%%%%%%%%%%%%%%%%%%%

\subsection{The Hamiltonian for point vortex dynamics}
Let $(M,\omega)$ be a symplectic manifold {and $N \in \mathbb{N}$. For $1 \leq k \leq N$, let $\tau_k: \Pi_{k=1}^N  M_k \to M$ be} the projection on the $k$th factor and let $\Gamma_1, \dots, \Gamma_N\in\R^{\neq 0}$.
Consider the space $\mathcal{M}:= \Pi_{k=1}^N  M \setminus \Diag_N(M)$ and endow it with the symplectic form $\Omega:=\Omega(\Gamma):= \sum_{k=1}^N \Gamma_k  \tau_k^* \omega$. 
Recall Green's function $G: (M \times M) \setminus \Diag_N(M) \to \R$ defined by the expression \eqref{Gr}
and define the so-called {\em Robin function} (see also Dritschel $\&$ Boatto \cite{dritschel2015motion})
\begin{align*}
R : M \to \R, \quad R(t) :=\lim_{\tilde{t} \to t}\left(G(\tilde{t} , t )-\frac{1}{2\pi}\log d( \tilde{t} , t ) \right).
\end{align*}
The Hamiltonian $H: \mathcal{M} \to \R$, 
\begin{align}
\label{ham}
	H(s_1, \dots, s_N): =\sum_{1\leq i < j\leq N} \Gamma_i\Gamma_jG(s_i,s_j)+\sum_{k=1}^N\Gamma_k^2R_g(s_k)
\end{align}
describes the dynamics of $N$ vortices with vortex strength $\Gamma_k\in\R^{\neq 0}$ for $k=1,\dots,N$ on the phase space $\mathcal M$. Green's function $G$ describes the interaction between pairs of distinct vortices and the Robin function takes self-interactions into account (for more details, see {Boatto \& Koiller \cite{Boatto2015}}).

To study the vortex dynamics on an explicitly given symplectic manifold, we need explicit formulas for Green's function. Unfortunately, Green's functions are explicitly known only for certain classes of manifolds as for instance planes, hyperbolic planes, and $2$-spheres, see Galajinsky \cite{galajinsky2022generalised}, Lim $\&$ Montaldi $\&$ Roberts \cite{lim2001relative}, Montaldi $\&$ Nava-Gaxiola \cite{montaldi2014point}.

Motivated by the study of the generic and degenerate coadjoint orbits of $\mathsf{SU}(3)$, we naturally are interested in the Green's functions on these coadjoint orbits. The degenerate coadjoint orbit is $\mathbb{CP}^2$ and, in the case of $\mathsf{SU}(n+1)$, one of the degenerate coadjoint orbits is $\mathbb{CP}^n$.

\subsection{Green's function {and the Hamiltonian on the coadjoint orbit $\mathbb{CP}^n$}}
The aim of this subsection is to obtain an explicit formula for Green's function on $\mathbb{CP}^n$. This will allow us then to write down the Hamiltonian for the point vortex dynamics on $\mathbb{CP}^n$ explicitly.

{Given an arbitrary compact Riemannian manifold, the} explicit computation of Green's function is not obvious. However, there are certain homogeneous spaces for which methods are available to obtain an explicit formula for Green's function. Among these spaces, there are compact rank one symmetric spaces, briefly CROSS spaces. All CROSS spaces are given by the following list: the $n$-sphere $\mathbb{S}^n$, the projective spaces $\mathbb{KP}^n$ for $\mathbb{K}\in\{\mathbb{R},\mathbb{C},\mathbb{H}\}$, and the octonionic plane $\mathbb{OP}^2$.

CROSS spaces are special examples of so-called {\em locally harmonic Blaschke manifolds} (see Besse \cite{besse1978manifolds}) for which the following result was established.

%\begin{definition}
%Let $M$ be a connected and compact Riemannian manifold. Then $M$ is \textit{Blaschke} in a point $p\in M$ if $\operatorname{inj}_M(p)=\operatorname{rad}_M(p)$. 
%\todo{\ssc{what is $\operatorname{rad}_M(p)$ ?}}	
%The manifold $M$ is said to be \textit{Blaschke} if it is Blaschke at at every point in $M$.
%\end{definition}

\begin{prop}[{Beltr\'an \& Corral \& Criado del Rey \cite{beltran2019discrete}}] 
Let $M$ be a locally harmonic Blaschke manifold and denote the geodesic distance between two points $x,y\in M$ by $r(x,y)=:r$. Then, Green's function on $M$ is given by $G(x,y)=\varphi(r)$ where $\varphi$ is determined by the differential equation 
\begin{align}
\label{diffequation}
	\varphi'(r) = - \frac{1}{r^{n-1} V_M(r)\vol(M)} \int_r^{\operatorname{inj}(M)}t^{n-1} V_M(t)dt
\end{align}
where $V_M$ is the so-called volume density and $\operatorname{inj}(M)$ is the injectivity radius of the manifold.
\end{prop}

\noindent 
Recall that the {\em diameter} of a Riemannian manifold $(M,g)$ is defined by 
\[ \diam(M):=\sup_{x,y\in M}r(x,y) \] 
where $r(x,y)$ is the geodesic distance between $x, y \in (M, g)$.
We now focus on the CROSS space $M=\mathbb{CP}^n$ where, in fact, the injectivity radius is equal to the diameter. The density function $V_{\mathbb{CP}^n}$ was explicitly determined by Kreyssig \cite{kreyssig2010introduction} as
\begin{align}
\label{volumeDensityCPn}
	V_{\mathbb{CP}^n}(r) = \frac{2^{2n-1}}{r^{n-1}}\sin^{2n-1}(r)\cos(r).
\end{align}

In order to solve the differential equation \eqref{diffequation} for $M= \mathbb{CP}^n$ we need the following technical result.

\begin{lemma}
\label{integraal}
	Let $n\in\mathbb{N}^{>0}$. Then
    \begin{align*}
	\int \frac{1-\sin^{2n}(x)}{\sin^{2n-1}(x)\cos(x)} \ dx = \log(\sin(x))-\sum_{j=1}^{n-1}\frac{1}{2j\sin^{2j}(x)}.
	\end{align*}
\end{lemma}

\begin{proof}
	We start with 
	\begin{align*}
	\int\frac{1-\sin^{2n}(x)}{\sin^{2n-1}(x)\cos(x) } \ dx &=\int \frac{1}{\sin^{2n-1}(x)\cos(x)} \ dx - \int\tan(x) \ dx 
	\\&= \int \frac{1}{\sin^{2n-1}(x)\cos(x)} \ dx  + \log(\cos(x)).
	\end{align*}
	Using $	1+\cot^2(x) = \csc^2(x)$, we obtain 
	\begin{align*}
	\int \frac{1}{\sin^{2n-1}(x)\cos(x)} \ dx & = \int \sec(x)\csc^{2n-1}(x) \ dx 
	= 	\int \sec(x)\csc^{2n-2}(x)\csc(x) \ dx 
	\\&=\int \sec(x)\csc(x)(1+\cot^2(x))^{n-1} \ dx.
	\end{align*}
Using the binomial formula 
\begin{align*} (1+\cot^2(x))^{n-1} = 
	\sum_{k=0}^{n-1}{n-1 \choose k}\cot^{2k}(x)
	\end{align*}
 we obtain 
 \begin{align*}
	\int \sec(x)\csc(x)\sum_{k=0}^{n-1}{n-1 \choose k}\cot^{2k}(x) \ dx = \sum_{k=0}^{n-1}{n-1 \choose k}\int \sec(x)\csc(x)\cot^{2k}(x)\ dx.
	\label{formulaSEC}
	\end{align*}
Since 
\begin{align*}
	\sec(x)\csc(x)\cot^{\ell}(x) = \frac{1}{\sin(x)\cos(x)}\frac{\cos^{\ell}(x)}{\sin^{\ell}(x)}=\frac{\cos^{\ell-1}(x)}{\sin^{\ell+1}(x)}=\frac{\cot^{\ell-1}(x)}{\sin^2(x)}
\end{align*}
the integral becomes
\begin{align*}
	\int \cot^{\ell-1}(x)\csc^2(x) \ dx.
\end{align*}
	By using the substitution $v=\cot(x)$ and $dv=-\csc^2(x)  dx $ we have 
	\[- \int v^{\ell-1} \ dv= -\frac{v^{\ell}}{\ell} + C =-\frac{\cot^{\ell}(x)}{\ell}+C \]
	where $C$ is the integration constant.
	Now, for $\ell=2k$, we can rewrite
	\[ \sum_{k=0}^{n-1}{n-1 \choose k}\int \sec(x)\csc(x)\cot^{2k}(x) \ dx= 
	\int \sec(x)\csc(x) \ dx 
	-\sum_{k=1}^{n-1}{n-1 \choose k}\frac{\cot^{2k}(x)}{2k}.\]
	We compute 
	\[ \int \sec(x)\csc(x)dx =\log(\sin(x)) - \log(\cos(x)).\]
	%Thus, splitting of this degenerate case we have the following result for our integral \begin{align*}
	%\int \sec(x)\csc^{2n-1}(x)dx =	\log(\sin(x))-\log(\cos(x)) -\sum_{k=1}^{n-1}{n-1 \choose k}\frac{\cot^{2k}(x)}{2k}.
	%\end{align*}
	This means for our original integral
	\begin{align*}
	& \int\frac{1-\sin^{2n}(x)}{\sin^{2n-1}(x)\cos(x)}dx \\
	& \quad =  	\log(\sin(x))-\log(\cos(x)) -\sum_{k=1}^{n-1}{n-1 \choose k}\frac{\cot^{2k}(x)}{2k} +\log(\cos(x))	\\
	& \quad =\log(\sin(x)) -\sum_{k=1}^{n-1}{n-1 \choose k}\frac{\cot^{2k}(x)}{2k}+{C,}
	\end{align*}
	{where $C$ is an integration constant.}
	By using repeatedly
	\[ \cot^2(x)=\frac{1}{\sin^2(x)}{-1}, \] 
	{and by choosing $C$ suitably we arrive at}
	\[	\int \frac{1-\sin^{2n}(x)}{\sin^{2n-1}(x)\cos(x)}dx=\log(\sin(x))-\sum_{j=1}^{n-1}\frac{1}{2j\sin^{2j}(x)}.\]
\end{proof}

We now obtain

\begin{theorem}
\label{greenFunctionCPn}
	Green's function on $\mathbb{CP}^n$ with the Fubini-Study metric is given by $G:\mathbb{CP}^n\times \mathbb{CP}^n \setminus \Diag_2(\mathbb{CP}^n ) \to \mathbb{R}$ with
	\begin{align*}
	%G:\mathbb{CP}^n\times \mathbb{CP}^n\to \mathbb{R}
	%,\quad
	G(\xi,\eta)= 
	-\frac{1}{2n\cdot \vol(\mathbb{CP}^n)}\left(\log(\sin(r(\xi,\eta)))-\sum _{j=1}^{n-1} \frac{1}{2j \sin^{2j}(r(\xi,\eta))}\right)
	\end{align*}
	where 
	$r(\xi,\eta) = \arccos\sqrt{\frac{\langle \xi,\eta\rangle \langle \eta,\xi\rangle_H}{\langle \xi,\xi\rangle \langle \eta,\eta\rangle_H}}$ is the geodesic distance between the two point in $\mathbb{CP}^n$ and $\langle \cdot,\cdot \rangle_H$ is the Hermitian inner product.
\end{theorem}

\begin{proof}
	Recall from Equation \eqref{volumeDensityCPn} that the volume density of $\mathbb{CP}^n$ is
	\begin{align*}{\label{ODE}}
	V_{\mathbb{CP}^n}(r) &= \frac{2^{2n-1}\sin^{2n-1}(r)\cos(r)}{r^{n-1}}
	\end{align*}
	and that 
	\[ \operatorname{inj} (\mathbb{CP}^n) =  \diam(\mathbb{CP}^n)=\frac{\pi}{2} \quad\text{ and }\quad \vol(\mathbb{CP}^n)=\frac{\pi^n}{n!}.\]
	Therefore the ODE for $\varphi(r)$ from \eqref{diffequation} can be written as 
	\[ \varphi'(r)=-\frac{1}{r^{n-1} V_{\mathbb{CP}^n} (r) \vol(\mathbb{CP}^n )}\int_r^{\diam(\mathbb{CP}^n)}t^{n-1} V_{\mathbb{CP}^n} (t) \ dt \]
	and gives rise to the following equation:
\begin{align*}
\varphi'(r)&=-\frac{1}{\vol(\mathbb{CP}^n)\sin^{2n-1}(r)\cos(r)}\int_r^{\frac{\pi}{2}}\sin^{2n-1}(t)\cos(t)dt \\&= -\frac{1}{\vol(\mathbb{CP}^n)\sin^{2n-1}(r)\cos(r)}\  \frac{1}{2n}(1-\sin^{2n}(r)) . 
\end{align*}
	Solving this ODE gives the following formula for the fundamental solution
	\begin{align*}
	\varphi(r)&=-\frac{1}{2n\cdot \vol(\mathbb{CP}^n)}\int \frac{1-\sin^{2n}(r)}{\sin^{2n-1}(r)\cos(r)} \ dr.
	%\\&=-\frac{1}{2n\cdot \vol(\mathbb{CP}^n)} \left(	\frac{\sin^{2(1-n)}(r) {}_{2}F_{1}(1,1-n;2-n,\sin^2(r))}{2(1-n)}	+\log(\cos(r))\right)	+C.
	\end{align*}
	Using Lemma \ref{integraal} now gives the result.
\end{proof}

\begin{theorem}
\label{fullhamiltonian_cpn}
The Hamiltonian for the $N$ point vortex dynamics on the projective space $\mathbb{CP}^n$ is explicitly given by
\begin{align*}
& H: \bigl(\mathbb{CP}^n\bigr)^N \setminus \Diag_N(\mathbb{CP}^n) \to \R, \\
& H (\zeta) = -\frac{1}{2(n-1)!\pi^n}\sum_{\alpha < \beta}^N \Gamma_{\alpha}\Gamma_{\beta}\left(\log(\sin(r(\zeta_{\alpha},\zeta_{\beta})))-\sum _{j=1}^{n-1} \frac{1}{2j \sin^{2j}({r(\zeta_{\alpha}, \zeta_{\beta})})}\right)
\end{align*}
where $\zeta = (\zeta_1, \dots, \zeta_N)$ and $r(\zeta_{\alpha},\zeta_{\beta})$ is the geodesic distance on $\mathbb{CP}^n$ between the two points given by 
\begin{align*}
r(\zeta_{\alpha},\zeta_{\beta}) = \arccos\sqrt{\frac{\langle \zeta_{\alpha},\zeta_{\beta}\rangle_H \langle \zeta_{\beta},\zeta_{\alpha}\rangle_H}{\langle \zeta_{\alpha},\zeta_{\alpha}\rangle_H \langle \zeta_{\beta},\zeta_{\beta}\rangle_H}},
\end{align*}
where $\langle \cdot,\cdot \rangle_H$ is the Hermitian inner product.
\end{theorem}

\begin{proof}
Equation \eqref{ham} gives the formal expression for the Hamiltonian of the N point vortex problem on a manifold. Using the explicit expression for the Green function on $\mathbb{CP}^n$ from Theorem \ref{greenFunctionCPn} yields the result.
\end{proof}

The Hamiltonian vector field $\mathcal{X}_H$ (and then also the equations of motion) for the Hamiltonian from Theorem \ref{fullhamiltonian_cpn} can be computed either by using the implicit formula involving the symplectic form or by making use of an compatible almost complex structure $J$ for the Fubini-Study metric, i.e., via $X^{H} = J\operatorname{grad}(H)$.

%%%%%%%%%%%%%%%%%%%%%%%%%%%%%%%%%%%%%%55
%%%%%%%%%% new subsection  %%%%%%%%%%%%%

\subsection{The Hamiltonian on the flag manifold $\mathbb{F}_{1,2}(\C^3)$}
\label{hamFlag}

As we saw in previous subsections, the explicit knowledge of Green's function is quite rare which makes it complicated to obtain an explicit expression for the Hamiltonian of the point vortex problem in many situations.

We do not yet have an explicit formula for Green's function and the Laplacian on the flag manifold $\mathbb{F}_{1,2}(\C^3)$ so that we also do not yet have an explicit expression for the Hamiltonian of the point vortex problem.
One idea to approach this open question may be the fibration 
$$
\mathbb{S}^2\longrightarrow W^6 \simeq \mathbb{F}_{1,2}(\C^3) \longrightarrow \C\mathbb{P}^2
$$
together with the hope to deduce Green's function on $W^6 \simeq \mathbb{F}_{1,2}(\C^3)$ from those on 
$\mathbb{S}^2$ and $\C\mathbb{P}^2$ and thus obtain the Hamiltonian on $W^6 \simeq \mathbb{F}_{1,2}(\C^3)$.

In fact, this poses the more general question of the behaviour of Green's function with respect to fibrations in general. But this is beyond the scope of the present paper.

  \section*{ Acknowledgments} 
\noindent The authors wish to thank Marine Fontaine for helpful discussions and useful comments.
 S.\ Hohloch was partially and G.\ Muarem was fully supported by the FWO-EoS project \textit{`Symplectic Techniques in Differential Geometry'} {G0H4518N}.  {The results in this paper were part of the PhD thesis of G.\ Muarem. The authors wish to stress that -- like it is common in mathematics -- the names are ordered alphabetically and not according to the effort.}
\bibliographystyle{alphaurl}
\bibliography{untitled}
\end{document}